\documentclass[11pt]{article}

\usepackage{amsmath}
\usepackage{amsfonts}
\usepackage{color}
\usepackage{float}   
\usepackage{pifont}  
\usepackage{amssymb}
\usepackage{graphicx,bm}
\usepackage{enumerate}
\usepackage{amsthm,amscd,nicefrac}
\usepackage[all]{xy}

 \allowdisplaybreaks

 \usepackage[colorlinks, linkcolor=red, anchorcolor=green, citecolor=blue]{hyperref}  

\newtheorem{theorem}{Theorem}[section]

\newtheorem{conjecture}[theorem]{Conjecture}

\newtheorem{lemma}[theorem]{Lemma}
\newtheorem{proposition}[theorem]{Proposition}

\newtheorem{problem}[theorem]{Problem}
\newtheorem{example}[theorem]{Example}

\begin{document}

\title{A spectral  Lov\'{a}sz--Simonovits theorem}

\author{Yongtao Li$^{1}$\quad 
 Lihua Feng$^{2,}$\thanks{Corresponding author. \\
E-mail addresses:  
\url{ytli0921@hnu.edu.cn} (Y. Li), 
\url{fenglh@163.com} (L. Feng), 
\url{ypeng1@hnu.edu.cn} (Y. Peng)} 
\quad  
Yuejian Peng$^{3}$  \\
{\small $^{1}$Yau Mathematical Sciences Center, Tsinghua University, Beijing, China.} \\ 
{\small $^{2}$School of Mathematics and Statistics,  Central South University, Changsha, China} \\ 
 {\small $^{3}$School of Mathematics, Hunan University, Changsha, China }  
 }


\maketitle

\vspace{-0.5cm}

\begin{abstract}
A fundamental result in extremal graph theory is attributed to Mantel's theorem, which states that every graph on $n$ vertices with more than $\lfloor n^2/4 \rfloor$ edges must contain a triangle. 
 Lov\'{a}sz and Simonovits (1975) provided a supersaturation phenomenon by showing that for any $q< n/2$, every graph with $\lfloor n^2/4 \rfloor +q$ edges contains at least $q\lfloor n/2 \rfloor$ triangles.  This result resolved a conjecture proposed by Erd\H{o}s in 1962. 
In this paper, we establish a spectral counterpart of 
the result of Lov\'{a}sz and Simonovits. 
Let $Y_{n,2,q}$  be the graph obtained 
from the bipartite Tur\'{a}n graph $T_{n,2}$ by embedding 
a matching with $q$ edges into 
the partite set of size $\lceil n/2\rceil$.  
Using the supersaturation-stability method 
and the spectral techniques,  
we firstly prove that for $q\le \frac{1}{11}\sqrt{n}$, 
every graph $G$ on $n$ vertices with spectral radius $\lambda (G) \ge \lambda (Y_{n,2,q})$ contains at least $q\lfloor n/2 \rfloor$ triangles.  
We also show that the bound $q=O(\sqrt{n})$ is tight up to a constant factor,  yielding a phenomenon different from that in edge supersaturation.  
Our result answers a spectral triangle counting problem 
proposed by Ning and Zhai (2023).  
Secondly, let $T_{n,2,q}$  be the graph obtained 
from  $T_{n,2}$ by embedding a star with $q$ edges into 
the partite set of size $\lceil n/2\rceil$.  
We show further that $T_{n,2,q}$ is the unique extremal graph 
that contains at most $q\lfloor n/2 \rfloor$ triangles and 
attains the maximum spectral radius. 
Thirdly,  we present an asymptotic spectral stability result 
 under a specific constraint on the triangle covering number. 
This result could be viewed as 
a spectral extension of a recent result proved by Balogh and Clemen (2023), 
and independently by Liu and Mubayi (2022).  
 \end{abstract}




\section{Introduction} 

The Tur\'{a}n type problem states that for a graph $F$, 
what is the minimum integer $m$ 
such that every $n$-vertex graph with more than 
 $m$ edges contains a copy of $F$. 
This extremal number  is denoted by $\mathrm{ex}(n,F)$. 
  It is a longstanding open problem in extremal combinatorics to develop some understanding of the function $\mathrm{ex}(n,F)$ 
  for a general graph $F$.  
A well-known result of Mantel 
(see \cite{Bollobas78,Lov1979}) asserts that 
an $n$-vertex graph with more than 
$\lfloor n^2/4\rfloor $ edges contains a triangle. 
The bound is tight by considering the bipartite Tur\'{a}n graph  $T_{n,2}$, which is a complete bipartite graph whose 
two parts have sizes as equal as possible. 
Shortly afterwards, 
there have been various generalizations of Mantel’s theorem; 
 see, e.g., \cite{BT1981,Bon1983}.

 \subsection{Extremal graph results on triangles}

About eighty years ago, 
Rademacher (unpublished, see Erd\H{o}s \cite{Erd1955,Erdos1964}) provided  
an interesting extension on Mantel's theorem by showing that  
there are not just one but at least $\lfloor {n}/{2}\rfloor$ triangles in 
any $n$-vertex graph with more than $\lfloor {n^2}/{4} \rfloor$ edges.

\begin{theorem}[Erd\H{o}s--Rademacher, 1941]  \label{thmrad}
If $G$ is a graph on $n$ vertices with 
$e(G)\ge  \lfloor {n^2}/{4} \rfloor +1$, 
then $G$ has at least $\lfloor {n}/{2}\rfloor$ triangles, 
with equality if and only if 
$G$ is the graph obtained from 
 $T_{n,2}$ by 
adding one edge into the partite set of size $\lceil {n}/{2}\rceil$.  
\end{theorem} 

The problem of counting triangles  is referred to as the Erd\H{o}s--Rademacher problem, 
which is regarded as a starting point of the supersaturation in graph theory.  
Erd\H{o}s \cite{Erdos1964} proved a variant  
for a graph with less than $\lfloor n^2/4\rfloor$ edges.  
The Erd\H{o}s result states that 
 if $\ell \ge 0$ and $ e(G) \ge  \lfloor {n^2}/{4} \rfloor - \ell$, 
 and $G$ has a triangle, 
then it has at least $\lfloor {n}/{2}\rfloor - \ell -1$ triangles. 
Setting $\ell =0$, 
we see that 
if $e(G)\ge \lfloor {n^2}/{4} \rfloor$, then 
$G$ has at least $\left\lfloor {n}/{2} \right\rfloor -1$ triangles, 
unless $G=T_{n,2}$.  
Moreover, 
Erd\H{o}s \cite{Erd1962a,Erd1962b}
 showed that there exists $c>0$ such that if $1\le q<cn$ and $n$ is sufficiently large, then
 every $n$-vertex graph with $\lfloor n^2/4\rfloor +q$
 edges  has at least $q\lfloor {n}/{2}\rfloor$ triangles.
This was further generalized by 
Lov\'{a}sz and Simonovits \cite{LS1975,LS1983}. 

\begin{theorem}[Lov\'{a}sz--Simonovits, 1975] 
\label{thm-LS1975}
Let $1\le q < {n}/{2}$ be an integer 
and $G$ be an $n$-vertex graph with 
$ e(G)\ge \lfloor {n^2}/{4}  \rfloor + q$.  
Then $G$ contains at least $q \lfloor {n}/{2}\rfloor $ 
triangles. 
\end{theorem}

A quick proof of Theorem \ref{thm-LS1975} in the range of $q\le {n}/{4}$ can be found in 
\cite[page 302]{Bollobas78}. Incidentally, 
we mention that 
the condition $q< {n}/{2}$ is necessary. 
For even $n$ and $q = {n}/{2}$, one can add $q +1$ extra edges 
to the larger vertex part of $K_{\frac{n}{2}+1, \frac{n}{2}-1}$ such that there is no triangle 
among these $q+1$ edges. 
Hence, we get a graph with ${n^2}/{4} +q$ edges and $(q +1)({n}/{2} -1) = 
{n^2}/{4} -1$ triangles, which are less than $q \lfloor {n}/{2}\rfloor$ triangles.

\medskip 
The stabilities of Theorems \ref{thmrad} and \ref{thm-LS1975} 
were recently studied by Xiao and Katona  \cite{XK2021}, 
Liu and Mubayi \cite{LM2022-Erd-Rad} and Balogh and 
Clemen \cite{BC2023}. 
A graph is called color-critical if it has an edge 
whose deletion reduces its chromatic number. 
For example, the complete graphs and odd cycles are color-critical. Apart from the triangles, 
the supersaturation problem for color-critical graphs was also 
investigated by Mubayi \cite{Mubayi2010}, Pikhurko and Yilma \cite{PY2017}, Ma and Yuan \cite{MY2023}. 
For related results, we refer to \cite{Rei2016,LPS2020} 
and references therein.

\subsection{Eigenvalues and counting triangles} 

Spectral graph theory aims to apply algebraic properties, such as the eigenvalues of associated matrices, to investigate the structural properties of a graph. 
In recent years, significant breakthroughs have been made in understanding graph structure through matrix eigenvalues; see, e.g.,   \cite{Huang2019,JTYZZ2021}.  
Let $G$ be a graph on $n$ vertices, denoted by $V(G)=\{v_1,\ldots ,v_n\}$. 
The adjacency matrix of $G$ is denoted by $A(G)=[a_{i,j}]_{i,j=1}^n$, 
where $a_{i,j}=1$ if $v_i$ and $v_j$ are adjacent, and $a_{i,j}=0$ otherwise. 
The spectral radius $\lambda (G)$ is defined as the maximum modulus 
of the eigenvalues of $A(G)$. By the Perron--Frobenius theorem, $\lambda (G)$ is the largest eigenvalue of $A(G)$ and there exists a non-negative eigenvector corresponding to it. The Rayleigh quotient gives 
$\lambda (G)=\max_{\bm{x}\in \mathbb{R}^n} \frac{\bm{x}^{\mathrm{T}}A(G)\bm{x}}{\bm{x}^{\mathrm{T}}\bm{x}}$, which implies $\lambda (G)\ge \frac{2m}{n} \ge \delta (G)$, 
where each equality holds if and only if $G$ is a regular graph.

\medskip 
The classical Tur\'{a}n theorem states that 
$e(G)>(1- \nicefrac{1}{r} )n^2/2$ forces a clique $K_{r+1}$. 
Similarly, one may wish to investigate the condition on spectral radius $\lambda (G)$ to guarantee certain substructures.  
Such problems fall under the umbrella of spectral extremal graph theory, which is one of the most 
popular topic in spectral graph theory. 
A well-known result on this topic was obtained 
in 1986 by Wilf \cite{Wil1986}, who showed that 
every $n$-vertex graph $G$ with 
$\lambda (G) > (1- \nicefrac{1}{r})n$ has  
a clique $K_{r+1}$. 
Another significant result due to Nikiforov \cite{Niki2002cpc} 
states that every $m$-edge graph $G$ with $\lambda^2 (G)>
 {(1- \nicefrac{1}{r})2m}$ contains a clique $K_{r+1}$. 
The above spectral results generalize 
the aforementioned Tur\'{a}n theorem. 
 Nikiforov's result implies Wilf's result 
by invoking Tur\'{a}n's bound or Rayleigh's bound $\lambda (G) \ge 2m/n$. 
Indeed,  if $\lambda (G)> (1- \nicefrac{1}{r})n$, 
then $\lambda^2(G)> (1-\nicefrac{1}{r})n \cdot (2m/n) 
= (1- \nicefrac{1}{r})2m$, as desired. 
In addition, Wilf's bound implies the ``strong'' Tur\'{a}n theorem. To see this, if $G$ is a $K_{r+1}$-free graph, 
then $\delta (G)\le \lambda (G) \le (1- \nicefrac{1}{r})n$ by assuming Wilf's bound. Hence $\delta (G)\le n - \lceil n/r\rceil = \delta (T_{n,r})$. 
Taking a vertex $u\in V(G)$ with degree $d(u) = \delta (G)$.  
By induction, we get $e(G\setminus \{u\}) \le e(T_{n-1,r})$, which yields $e(G)= e(G\setminus \{u\}) + \delta (G)  
\le e(T_{n-1,r}) + \delta (T_{n,r}) = e(T_{n,r})$, as expected.

\medskip 
The spectral Tur\'{a}n problems have been receiving considerable 
attention in the last two decades 
and it is an important topic in modern graph theory; see, e.g., 
\cite{Niki2002cpc,BN2007jctb,LP2022second} for cliques, 
\cite{FYZ07}  for matchings, 
\cite{CDT2023-even-cycle} for even cycles,  
\cite{LN2023,Zhang2024,LZS2024} for cycles of consecutive lengths, 
\cite{Niki2009ejc} for color-critical graphs, 
\cite{BG2009,Niki2010laa} for complete bipartite graphs, 
\cite{TT2017,LN2021outplanar} for planar and 
outerplanar graphs, 
\cite{Tait2019,ZL2022jctb,CLZ2024,ZFL2024} for graphs without certain minors, 
\cite{Niki2009cpc} for a spectral Erd\H{o}s--Stone--Bollob\'{a}s theorem,  
\cite{Niki2009jgt} for the spectral stability theorem, 
\cite{CDT2023} for a spectral Erd\H{o}s--S\'{o}s theorem, 
\cite{ZL2022} for a spectral Erd\H{o}s--P\'{o}sa theorem, 
\cite{LLP2024-AAM} for a spectral Erd\H{o}s--Rademacher theorem, 
\cite{LFP2024-triangular} for a spectral Erd\H{o}s--Faudree--Rousseau theorem, 
\cite{CFTZ20,ZLX2022,LFP2024-triangular} for friendship graphs, 
\cite{2022LLP, LFP-count-bowtie} for bowties, 
\cite{ZL2022jgt,Niki2021} for book graphs, 
\cite{CDT2022} for odd wheels  
and \cite{Wang2022,FTZ2024} for a characterization of spectral-consistent graphs. 

\medskip 
Although there has been a wealth of research 
on the spectral extremal problems, 
there are \textit{very few} results on the spectral problem  
of counting substructures. 
In 2007, Bollob\'{a}s 
and Nikiforov \cite{BN2007jctb} proved that every $n$-vertex graph $G$ has at least 
$\big( \frac{\lambda (G)}{n} - 1 + \frac{1}{r} \big) 
\frac{r(r-1)}{r+1} ( \frac{n}{r})^{r+1}$ copies of $K_{r+1}$. 
This result provides a spectral supersaturation for cliques. 
The spectral supersaturation for general graphs (see \cite[Sec. 5]{LLP2024-AAM}) states that  
for any $\varepsilon >0$ and any graph $F$ with $\chi (F)=r+1$, there exist $\delta>0$ and $n_0$ such that if $G$ is a graph on $n\ge n_0$ vertices with  $ \lambda (G)\ge \left(1- \frac{1}{r} + \varepsilon \right)n$, 
 then $G$ contains at least $\delta n^{|F|}$ copies of $F$. 

\medskip 
Recall in Theorem \ref{thmrad} that $G$ has at least 
$ \lfloor {n}/{2} \rfloor$ triangles when $e(G) > e(T_{n,2})$.  Correspondingly, 
it is natural to consider the  spectral version:  
if $ \lambda (G)> \lambda (T_{n,2})$, 
whether $G$ has at least  $\lfloor {n}/{2}\rfloor$ triangles.   
However, this intuition is not true. 
Let $K_{a,b}^+$ be the graph obtained from $K_{a,b}$ by adding an edge to 
the vertex set of size $a$. 
For even $n$, taking $a=\frac{n}{2} +1$ and $b=\frac{n}{2}-1$,  
we compute that $\lambda (K_{\frac{n}{2}+1,\frac{n}{2}-1}^+) > \lambda (T_{n,2})$,   while  $K_{\frac{n}{2}+1,\frac{n}{2}-1}^+$ has exactly $\frac{n}{2}-1$ triangles.   
Ning and Zhai \cite{NZ2021} proved that 
$\lfloor {n}/{2}\rfloor -1 $ triangles can be guaranteed.

\begin{theorem}[Ning--Zhai, 2023] \label{thmNZ2021}
If $G$ is a graph on $n$ vertices with 
\[   \lambda (G) \ge \lambda (T_{n,2}), \]  
then  $G$ has at least 
$\left\lfloor {n}/{2}\right\rfloor -1 $ triangles, 
unless $G$ is the bipartite Tur\'{a}n graph $T_{n,2}$.  
\end{theorem}

Ning and Zhai \cite{NZ2021} 
 proposed a problem on finding spectral counting results 
 for color-critical graphs, including, e.g., the spectral version of Theorem \ref{thm-LS1975}. 
Interestingly, this problem was also selected as one of the ``Unsolved Problems'' in spectral graph theory \cite[Sec. 9]{LN2023-unsolved}. 
There are  few spectral results on counting substructures in a graph with given order, 
including only triangles \cite{NZ2021}, cliques \cite{BN2007jctb}, 
triangular edges \cite{LFP2024-triangular} 
and bowties \cite{LFP-count-bowtie}. 
The main contribution of this paper is to establish
 a spectral version of Theorem \ref{thm-LS1975}. 
To some extent, 
our result could be viewed as a partial 
resolution of the Ning--Zhai problem.

\section{Main results}

\subsection{A spectral  Lov\'{a}sz--Simonovits theorem}

Before trying to obtain the spectral counterpart of Theorem \ref{thm-LS1975}, the first obstacle we may encounter is how to formulate a spectral hypothesis that corresponds to the conventional assumption on the size of a graph. 
More precisely, 
we would like to determine an increment function $\delta(q)>0$ 
such that the condition $e(G)\ge \lfloor n^2/4 \rfloor +q$ can be replaced with a spectral condition of the form $\lambda (G)\ge \lambda (T_{n,2}) + \delta (q)$.  
For instance, is it sufficient to consider $\delta (q) =2q/n$?   
Since every additional edge guarantees $\lfloor n/2 \rfloor$ additional triangles, it is unlikely that 
such spectral increment can create $\lfloor n/2 \rfloor$ 
new triangles. 
From this point of view, it seems difficult to measure the increment on the spectral radius of a graph. 
To get around the difficulty, 
Li, Lu and Peng \cite{LLP2024-AAM} presented a new way to measure the increment on the spectral radius.    
Let $T_{n,2,q}$ be the graph obtained from
$T_{n,2}$ by embedding a star $K_{1,q}$ into the larger vertex part; see Fig. \ref{fig-thm2-3}.  
Under this notation, they define the spectral increment 
as $\delta (q) = \lambda (T_{n,2,q}) - \lambda (T_{n,2})$,  and proposed the following problem.  

\begin{problem}[Li--Lu--Peng, 2024]  
\label{conj-spec-LS}
Let $1\le q < {n}/{2}$ and $G$ be an $n$-vertex graph with
\[  \lambda (G)\ge \lambda (T_{n,2,q}) . \]
Then $G$ has at least  $ q \lfloor {n}/{2}\rfloor $ triangles, 
with equality if and only if $G=T_{n,2,q}$.
\end{problem}

Problem \ref{conj-spec-LS} could be regarded as 
a spectral version of the Lov\'{a}sz--Simonovits theorem. 
Especially,  
Li, Lu and Peng \cite{LLP2024-AAM}  
confirmed this problem in the initial case of $q=1$.

\medskip 
 In the present paper, 
 our goal is to confirm Problem \ref{conj-spec-LS} for any given $q\in \mathbb{N}$.  
Let $t(G)$ be the number of triangles in $G$. 
We show that 
$t(G)\ge q\lfloor {n}/{2}\rfloor$ in Problem \ref{conj-spec-LS} 
can be guaranteed by a slightly {\it weaker} condition. 
Let $Y_{n,2,q}$ be the graph obtained from 
$T_{n,2}$ by embedding $q$ pairwise disjoint edges 
into the partite set with larger size; see Fig. \ref{fig-thm2-3}.  
It is not difficult to verify that 
$\lambda (Y_{n,2,q}) < \lambda (T_{n,2,q})$ for every $q\ge 2$; see Proposition \ref{prop-Y-less-T}. 

  \begin{figure}[H]
\centering
\includegraphics[scale=0.9]{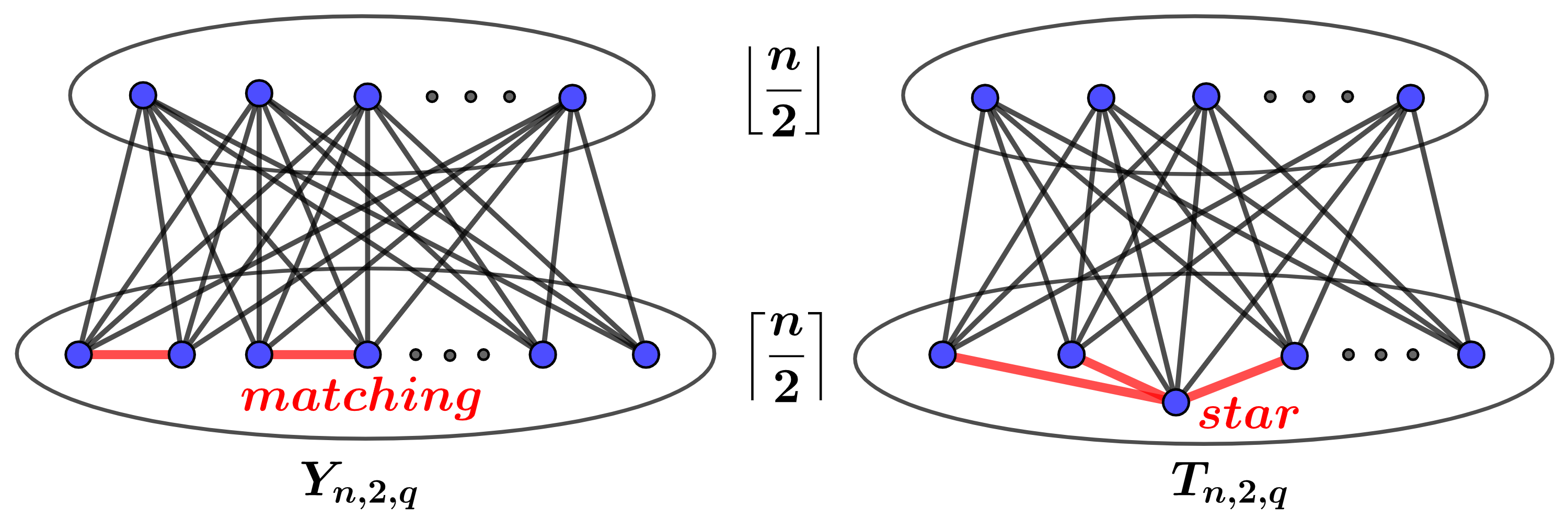} 
\caption{The graphs $Y_{n,2,q}$ and $T_{n,2,q}$.}
\label{fig-thm2-3}
\end{figure}

In what follows,  we give a stronger result than Problem \ref{conj-spec-LS}. 

\begin{theorem} \label{thm-Yn2q}
If $1\le q \le \frac{1}{11}\sqrt{n}$ and $G$ is an $n$-vertex graph with
\[  \lambda (G)\ge \lambda (Y_{n,2,q}), \]  
then there are at least  $ q \lfloor {n}/{2}\rfloor $ triangles in $G$.  
\end{theorem}

\noindent 
{\bf Remark.} 
The above bound $q\le \frac{1}{11}\sqrt{n}$ is {\it tight} up to a constant factor; see Example \ref{exam}. This reveals a {\it different} phenomenon from that of the traditional edge supersaturation. 
Moreover, the spectral condition in Theorem \ref{thm-Yn2q} is weaker than the size condition in Theorem \ref{thm-LS1975}, since $e(G)> e(Y_{n,2,q})$ implies $\lambda (G)> \lambda (Y_{n,2,q})$. 
Indeed, suppose that $\lambda (G)\le \lambda (Y_{n,2,q})$. Then, invoking the fact that $\lambda (G)\ge \frac{2e(G)}{n}$, we obtain $e(G)\le \lfloor \frac{n}{2} \lambda (G)\rfloor 
\le \lfloor \frac{n}{2} \lambda (Y_{n,2,q})\rfloor = e(Y_{n,2,q})$, 
where the last equality follows from $\lambda (Y_{n,2,q}) < \frac{n}{2} + \frac{2q}{n} + \frac{4q}{n^2}$ (see Lemma \ref{lem-LLP}).  
 In conclusion, the spectral condition is indeed weaker than the size condition. 
In addition, there are many graphs $G$ satisfying $\lambda (G)> \lambda (Y_{n,2,q})$ while $e(G)< e(Y_{n,2,q})$. 
From this perspective, the spectral condition of Theorem \ref{thm-Yn2q} has {\it wider} applicability.

The Bollob\'{a}s--Nikiforov bound  \cite{BN2007jctb} 
states that $t(G)\ge \frac{n^2}{12} \left( \lambda (G)- \frac{n}{2} \right)$.  This provides a lower bound for $t(G)$ under the assumption of Theorem \ref{thm-Yn2q}.  
For even $n$, we have $\lambda(G) \ge \lambda (Y_{n,2,q}) > \frac{n}{2} + \frac{2q}{n}$, and then $t(G) > \frac{q}{6}n$; For odd $n$, we obtain $\lambda(G) > \frac{n}{2}  - \frac{1}{2n}+ \frac{2q}{n}$, which implies $t(G) > \frac{4q-1}{24}n$. Since $\lambda(Y_{n,2,q}) < \frac{n}{2} + \frac{2q}{n} + \frac{4q}{n^2}$ for all $n\ge 4q$ (see Lemma \ref{lem-LLP}), the lower bound on $t(G)$ provided by Bollob\'{a}s--Nikiforov's bound does not exceed $\frac{q}{6}n + \frac{q}{3}$. In contrast, Theorem \ref{thm-Yn2q} offers a {\it stronger} bound $t(G) \geq q\lfloor \frac{n}{2} \rfloor$.

\medskip 
Recall that $\lambda (T_{n,2,q}) > \lambda (Y_{n,2,q})$. 
Applying Theorem \ref{thm-Yn2q}, 
we know that the desired bound 
in Problem \ref{conj-spec-LS} holds 
immediately. 
Using the double-eigenvector technique, we shall confirm Problem  \ref{conj-spec-LS} in the range $q\le \frac{1}{11}\sqrt{n}$, 
showing that $T_{n,2,q}$ is the unique graph that 
attains the maximum spectral radius among all $n$-vertex 
graphs 
with at most $q\lfloor {n}/{2}\rfloor$ triangles. 

\begin{theorem} \label{thm-sp-LS}
If  $1\le q \le \frac{1}{11}\sqrt{n}$ and $G$ is an $n$-vertex graph with
\[  \lambda (G)\ge \lambda (T_{n,2,q}), \]  
then $G$ has at least  $ q \lfloor {n}/{2}\rfloor $ triangles, 
with equality if and only if $G=T_{n,2,q}$. 
\end{theorem}

\noindent 
{\bf Remark.}  When embedding  
a graph $H$ with $q$ edges into one partite set 
of $T_{n,2}$, we can verify that 
the spectral radius of the resulting graph 
decreases according to the following order of $H$: 
star, complete graph, complete bipartite graph (that is not a star), cycle, path, and matching. 
Moreover, it can be verified that 
if we embed two matchings of sizes $q_1$ and $q_2$ 
(with $q_1+q_2=q$) into two partite sets, respectively, 
then the resulting graph has larger spectral radius than 
 $Y_{n,2,q}$. In contrast, 
 if we embed two stars of sizes $q_1$ and 
$q_2$ into two partite sets, respectively, then 
the resulting graph has smaller spectral radius  than  
$T_{n,2,q}$.

\subsection{A spectral extension of a stability result}

 The {\it triangle covering number}  $\tau_3(G)$ is the minimum number of vertices in a set that contains at least one vertex of each triangle in $G$. 
In particular,  $\tau_3(G)=1$ if and only if all triangles of $G$ 
share a common vertex. 
Xiao and Katona \cite{XK2021} gave a stability of Theorem \ref{thmrad} by showing  that if $e(G)> \lfloor n^2/4\rfloor$ and $\tau_3(G)\ge 2$, then $G$ has at least 
$n-2$ triangles. 
Moreover, they proposed a problem to study the stability of  Theorem \ref{thm-LS1975}. 
Recently, this problem was solved by Liu and Mubayi 
\cite{LM2022-Erd-Rad}, as well as Balogh and Clemen 
 \cite{BC2023} independently.

 \medskip 
 Before stating their result, we need to define two extremal graphs. 
 Firstly, 
let $G_1$ be the graph of order $n$, where $V(G_1)=A\cup B$ with $|A| = \lceil {n}/{2}\rceil +a$ and  
$|B| = \lfloor {n}/{2} \rfloor -a$, where $a\in \mathbb{N}$ is chosen later.  
Pick $2(s-1)$ vertices $x_1,y_1,\ldots ,x_{s-1},y_{s-1}$ in $A$ 
and two vertices $u_1,u_2$ in $B$. Adding $s$ edges 
$\{u_1,u_2\}$ and $\{x_i,y_i\}$, for $i\in \{1,\ldots ,s-1\}$  
into $K_{|A|,|B|}$, and  delete  $\alpha$ edges 
$\{u_1,x_1\}, \ldots ,\{u_1,x_{\alpha}\}$, 
where $\alpha := s-t- a^2 - \mathbf{1}_{\{2\nmid n\}} a$.   
Then $e(G_1)=\lfloor n^2/4\rfloor +t$ and $\tau_3(G_1)=s$. 
Secondly, 
let $G_2$ be the graph with the same vertex partition. 
Unlike the previous construction, we now 
pick $2s$ vertices 
$x_1,y_1,\ldots ,x_s,y_s\in A$ and one vertex $u\in B$, 
add the $s$ edges $\{x_i,y_i\}$, for $i\in \{1,\ldots ,s\}$ into $A$, 
and delete the $\alpha$ edges 
$\{u,x_1\},\ldots ,\{u,x_{\alpha}\}$ between $A$ and $B$. 
Similarly, we have $e(G_2)=\lfloor n^2/4\rfloor +t$ and $\tau_3(G_2)=s$. 
It is easy to see that 
$t(G_1)=(s-1)|B| + |A| -2 \alpha$ and 
$t(G_2)=s|B| - \alpha$. 
Now, we choose an appropriate $a\in \mathbb{N}$ 
to minimize the number of triangles in the above constructions.

\begin{theorem}[Balogh--Clemen, 2023]  \label{thm-BC}
Let $t,s\in \mathbb{N}$ such that 
 $0<t<s$. 
Suppose that 
$G$ is an $n$-vertex graph with $\lfloor n^2/4\rfloor +t$ edges.  If  $\tau_3(G) \ge s$ and $n$ is sufficiently large, then $G$ contains at least $\min\{t(G_1),t(G_2)\}$ triangles.
\end{theorem} 

An independent work was also studied by Liu and Mubayi \cite[Theorem 1.4]{LM2022-Erd-Rad}. 
In other words, there are two ways to embed 
$s$ disjoint edges in the partition 
in order to minimize the number of triangles, that is, 
embedding $s-1$ or $s$ edges in the larger vertex part, and other edges in the smaller part.   
Recently, Li, Feng and Peng \cite{LFP-count-bowtie} 
presented a spectral version of the result of Xiao and Katona \cite{XK2021} 
by showing that if $G$ is an $n$-vertex graph with 
$\lambda (G)\ge \lambda (T_{n,2})$ and $\tau_3(G)\ge 2$, 
then $G$ contains at least $n-3$ triangles. 
This bound is the best possible and provides a spectral stability of the Ning--Zhai result in Theorem \ref{thmNZ2021}. 

\medskip 
Next, we shall consider the general case  $\tau_3(G)\ge s$ for every $s\in \mathbb{N}$. 
The third main result in this paper reveals an asymptotically spectral extension of Theorem \ref{thm-BC}.  

\begin{theorem} \label{spectral-BC}
If $G$ is a graph with $n\ge 28s^2$ vertices with
 $\lambda (G)\ge \lambda (T_{n,2})$ and $\tau_3(G)\ge s$, 
 then there are at least $\frac{1}{2} sn - 5s^2$ triangles in $G$. 
\end{theorem}

\noindent 
{\bf Remark.}  
Theorem \ref{spectral-BC} gives the stability of Theorem \ref{thm-sp-LS}. 
Applying the fact $\lambda (G)\ge \frac{2e(G)}{n}$, we see that 
$e(G)> e(T_{n,2})$ implies $\lambda (G) > \lambda (T_{n,2})$. 
As a by-product of Theorem \ref{spectral-BC}, 
we find that if $e(G)> e(T_{n,2})$ and $\tau_3(G)\ge s$, 
then $G$ contains $\frac{1}{2}sn - O(1)$ triangles. 
This result could be viewed as an asymptotic version of 
Theorem \ref{thm-BC}, since $t(G_1) \approx t(G_2) \approx \frac{1}{2}sn - O(1)$. 
With additional efforts,  the constant term of Theorem \ref{spectral-BC} can be slightly  improved, 
but it seems difficult to establish the exact bound and  
determine the spectral extremal graphs corresponding to the extremal graphs $G_1$ and $G_2$ in Theorem \ref{thm-BC}.

\subsection{Our approach and organization}

{\bf Our approach.} 
The main idea in our proofs of Theorems \ref{thm-Yn2q}, \ref{thm-sp-LS} and \ref{spectral-BC} 
uses a simple yet surprisingly powerful technique known as 
the supersaturation-stability method.  
Informally, every graph with few triangles can be made bipartite by deleting a relatively small number of edges.   
Here, ``relatively small'' means a constant number of edges independent of $n$; see Theorem \ref{thm-tri-effi}.
This is different from the classical stability method, 
which typically requires the removal of roughly $\varepsilon n^2$ edges for some small $\varepsilon >0$. 
Although this description may sound somewhat vague, we will make it more precise in Section \ref{sec4}. 
Unlike the traditional stability method, 
our method requires only the condition $n \ge 121q^2$, instead of requiring that $q$ is fixed and $n$ is sufficiently large (or a tower-type function in $q$). 
This is because our proof avoids the use of the celebrated triangle removal lemma\footnote{The lemma says that for each $\varepsilon >0$,  there is $\delta >0$ such that every graph on $n$ vertices with at most $\delta n^3$ triangles can be made triangle-free by removing at most $\varepsilon n^2$ edges.  That is, a graph with a subcubic number of triangles can be made triangle-free by removing a subquadratic number of edges; see \cite{CF2013}.}. 
Consequently, $q$ is allowed to vary over a wide range and may even grow with $n$. 
Hence, we significantly improve the lower bound 
on the order $n$ of a graph with respect to $q$. 
We believe that the supersaturation-stability method used in this paper may be of independent interest and could potentially find further applications elsewhere, much like the classical stability approach, which has now seen widespread use. 
Last but not least, in addition to the supersaturation-stability method, we also employ the spectral technique developed by Cioab\u{a}, Feng, Tait and Zhang \cite{CFTZ20},  
 as well as the double-eigenvector technique 
 applied by Fang, Tait and Zhai \cite{FTZ2024}.

\medskip
\noindent 
{\bf Organization.}   
 In Section \ref{sec3}, we show the tightness of $q$ in Theorem \ref{thm-Yn2q} and introduce preliminaries for counting  triangles. 
As a warm up, we give a spectral stability result  
 for graphs with few triangles; 
 see Theorem \ref{thm-tri-effi}. 
 In Section \ref{sec4}, 
we outline the main ideas of our method and 
then present the proof of 
Theorem \ref{thm-Yn2q}. 
In Section \ref{sec5}, we employ the double-eigenvector technique to prove Theorem \ref{thm-sp-LS}. 
The proof of Theorem \ref{spectral-BC} is provided  
in Section \ref{sec6}.   
Finally, 
we propose several related spectral extremal problems  for  readers.

\medskip
\noindent 
{\bf Notation.}  
We write $G=(V,E)$ for a graph with vertex set 
$V=\{v_1,\ldots ,v_n\}$ and edge set $E=\{e_1,\ldots ,e_m\}$, where we admit $|V|=n$ and $|E|=m$. 
If $S\subseteq V$ is a subset of the vertex set, then 
$G[S]$ denotes the subgraph of $G$ induced by $S$, 
i.e., the graph on $S$ whose edges are those edges of $G$ 
with both endpoints in $S$. 
For abbreviation, we denote $e(S)=e(G[S])$.  
We write $G[S,T]$ for the induced 
subgraph of $G$
whose edges have one endpoint in $S$ and the other in $T$,  and we write $e(S,T)=e(G[S,T])$.  
We write $K_{s,t}$ 
for the complete bipartite graph with parts of sizes $s$ 
and $t$. 
Let $N(v)$ be the set of vertices adjacent to a vertex $v$ and 
let $d(v)=|N(v)|$. 
For simplicity, we denote $N_S(v)= N(v) \cap S$ 
and $d_S(v)=|N_S(v)|$. 
 We write $\lambda (G)$ for the spectral radius of $G$ 
 and $t(G)$ for the number of triangles in $G$.  
 Sometimes, we write $\lambda$ and $t$ for short. 
 The triangle covering number $\tau_3(G)$ is defined as the minimum number of vertices 
that hit all triangles of $G$. 
For a vertex set $S\subseteq V(G)$, 
we write $G\setminus S $ for the subgraph of $G$ 
by deleting all vertices of $S$ and their incident edges. 
Let $\bm{x}=(x_1,x_2,\ldots ,x_n)^{\mathrm{T}}$ be a Perron vector of $G$, 
i.e., a non-negative eigenvector corresponding to $\lambda (G)$. 
By scaling, we may further assume that $\max\{x_v: v\in V(G)\}=1$. 
For a subset $S\subseteq V(G)$, we denote $x_S=\sum_{v\in S}x_v$. 
In particular, we have $x_V= \sum_{v\in V(G)} x_v$.

\section{Preliminaries}

\label{sec3}

Recall that $Y_{n,2,q}$ is a graph obtained from the bipartite Tur\'{a}n graph  $T_{n,2}$ by embedding  $q$ pairwise disjoint edges into the partite set of size $\lceil n/2\rceil$. In what follows, we give an estimate on the spectral radius of $Y_{n,2,q}$. 

\begin{lemma} \label{lem-LLP}
(a) If $n$ is even, then $\lambda(Y_{n,2,q})$
is the largest root of
\[  f_1(x)=x^3-x^2 - (n^2x)/4 + n^2/4 - q n. \]
(b) If $n$ is odd, then $\lambda(Y_{n,2,q})$
is the largest root of
\[ f_2(x)=x^3 - x^2 - (n^2x)/4 + x/4 + n^2/4 -1/4 + q - q n. \]
Consequently, we have $ \sqrt{\left\lfloor {n^2}/{4} \right\rfloor +2q} < \lambda (Y_{n,2,q}) < \frac{n}{2} + \frac{2q}{n} + \frac{4q}{n^2}$ for every $n\ge 4q$. 
\end{lemma}

\begin{proof}
(a) Let $\bm{x}=(x_1,x_2,\ldots ,x_n)^{\mathrm{T}}$ be the Perron eigenvector of $\lambda(Y_{n,2,q})$.
Assume that $V(Y_{n,2,q} )=
X_1 \cup X_2 \cup Y$, where $X_1=\{u_1,v_1,u_2,v_2,\ldots,u_q,v_q \}$ is the set of  vertices of the matching of size $q$, $X_1\cup X_2$ and $Y$ are partite sets of $K_{\frac{n}{2}, \frac{n}{2}}$.
Since two vertices with the same neighborhood have the same value in corresponding coordinates of $\bm{x}$,  we may assume that
$x_{u_i}=x_{v_i}=a$ for $1\le i\le q$, $x_u=b$ for each $u\in X_2$, and $x_w=c$
for each $w\in Y$. Then
\[  \begin{cases}
\lambda a= a + \frac{n}{2} c, \\
\lambda b = \frac{n}{2}c, \\
\lambda c = 2qa + (\frac{n}{2}-2q)b.
\end{cases} \]
Thus, $\lambda(Y_{n,2,q})$ is the largest eigenvalue of
\[  B_1 =  \begin{bmatrix}
1 & 0 & \frac{n}{2} \\
0 & 0 & \frac{n}{2} \\
2q & \frac{n}{2}-2q & 0
\end{bmatrix}.  \]
By calculation, we know that $\lambda(Y_{n,2,q})$ is the largest  root of
\[ f_1(x) := \det (xI - B_1 ) 
=x^3 - x^2 - (n^2 x)/4 + n^2/4 - q n.  \]


(b) Let $V(Y_{n,2,q}) = X_1\cup X_2\cup Y$ be a partition, 
where $X_1=\{u_1,v_1,\ldots ,u_q,v_q\}$ spans a matching with 
$q$ edges, $X_1\cup X_2$ and $Y$
are partite sets of $K_{\frac{n+1}{2}, \frac{n-1}{2}}$
satisfying $|X_1| + |X_2|= \frac{n+1}{2}$
and $|Y|=\frac{n-1}{2}$. A similar argument yields that
$\lambda (Y_{n,2,q})$ is the largest eigenvalue of
\[  B_2= \begin{bmatrix}
1 & 0 & \frac{n-1}{2} \\
0 & 0 & \frac{n-1}{2} \\
2q & \frac{n+1}{2}- 2q & 0
\end{bmatrix}.  \]
Thus, $\lambda (Y_{n,2,q})$
is the largest root of
\[  f_2(x) := \det (xI - B_2)
=x^3 - x^2 + x/4 - (n^2 x)/4 + n^2/4 -1/4 - q n + q.  \] 
Moreover, it is easy to check that 
\[ f_1(\sqrt{n^2/4 +2q}) = 2q \sqrt{n^2/4 +2q} -qn - 2q <0  \]
and 
\[ f_2(\sqrt{(n^2-1)/4 +4}) 
=2q \sqrt{(n^2-1)/4 +2q} - qn -q <0. \]
Consequently, we get 
$ \sqrt{\lfloor n^2/4 \rfloor  +4} < 
\lambda (Y_{n,2,q}) $. Similarly, we can verify that $f_1(\frac{n}{2} + \frac{2q}{n} + \frac{4q}{n^2}) >0$ and $f_2(\frac{n}{2} + \frac{2q}{n} + \frac{4q}{n^2}) >0$, which implies that 
$\lambda (Y_{n,2,q}) < \frac{n}{2} + \frac{2q}{n} + \frac{4q}{n^2}$, as desired. 
\end{proof}

Similar to the proof of Lemma \ref{lem-LLP}, 
we prove the following proposition.  
 
\begin{proposition} \label{prop-Y-less-T}
For any $q\ge 2$, we have $\lambda (Y_{n,2,q}) < \lambda (T_{n,2,q})$. 
\end{proposition}

\begin{proof}
We start the proof in two cases. 
(a) If $n$ is even, by a simple computation as in the proof of Lemma \ref{lem-LLP}, we can obtain that $\lambda (T_{n,2,q})$ is the largest root of 
\begin{align*}  h_1(x) &:= x^4 - x^3 - (n^2 x^2)/4 - q x^2 + (n^2 x)/4  - (n x)/2  - n q x \\ 
& \quad  -(n q)/2 + (n^2 q)/4 - (n q^2)/2  .  
\end{align*}
Let $f_1(x)$ be the polynomial defined in Lemma \ref{lem-LLP}. It follows that 
\[  x f_1(x) - h_1(x) = q x^2 + (n x)/2 + (n q^2)/2 + (n q)/2 - (n^2 q)/4.  \]
 It is easy to see that for every $x> \frac{n}{2}$, we have 
 \[  x f_1(x) - h_1(x) > \tfrac{n}{2}f_1(\tfrac{n}{2}) - h_1(\tfrac{n}{2}) 
 =  \tfrac{1}{4}  (n^2 + 2 q (1 + q)n )> 0. \]
We conclude that $x f_1(x) > h_1 (x)$ for every $x> \frac{n}{2}$. 
Note that $\lambda (Y_{n,2,q})$ is the largest root of $x f_1(x)$. 
Combining with $\lambda (Y_{n,2,q}), \lambda (T_{n,2,q}) > \frac{n}{2}$, 
we obtain $\lambda (Y_{n,2,q}) < \lambda (T_{n,2,q})$. 

(b) If $n$ is odd, then $\lambda (T_{n,2,q})$ is the largest root of 
\begin{align*}
h_2(x) &:= x^4 - x^3 + x^2/4 - (n^2 x^2)/4 - q x^2 + x/4 - (n x)/2 + (n^2 x)/4 + q x - n q x \\
&\quad + q/4 - (n q)/2 + (n^2 q)/4 + q^2/2 - (n q^2)/2 . 
\end{align*}
Let $f_2(x)$ be the polynomial defined in Lemma \ref{lem-LLP}. Then 
\[  x f_2 (x) - h_2(x) = q x^2 + ( n-1) x/2 
 + ( n q)/2 - (n^2 q)/4 +  ( n-1) q^2/2 -q/4 .   \]
 For every $x> \frac{n-1}{2}$, it follows that 
 \[ x f_2 (x) - h_2(x) > \tfrac{n-1}{2}f_2(\tfrac{n-1}{2}) - h_2(\tfrac{n-1}{2}) 
 = \tfrac{1}{4} (n-1) (2 q^2 + n -1) >0.
   \]
In this case, observe that $\lambda (Y_{n,2,q}), \lambda (T_{n,2,q}) > \frac{n-1}{2}$, so we get $\lambda (Y_{n,2,q}) < \lambda (T_{n,2,q})$. 
\end{proof}

\subsection{The tightness of $q$ in Theorem \ref{thm-Yn2q}}

In this section, we provide an example showing that the bound of $q$ in Theorem \ref{thm-Yn2q} is tight up to a constant factor. 
Recall that $Y_{n,2,q}$ and $T_{n,2,q}$ are obtained from $T_{n,2}$ by adding a matching, and a star with $q$ edges into the partite set of size $\lceil n/2\rceil$, respectively.  

\begin{example} \label{exam}
If $q=0.8\sqrt{n}$ and $n\ge 110$ are integers, then $  \lambda (T_{n,2,q-1}) > \lambda (Y_{n,2,q})$. 
However, the graph $T_{n,2,q-1}$ contains exactly $(q-1)\lfloor n/2 \rfloor$ triangles. 
\end{example}

\begin{proof} 
We partition the proof into two cases. 
(a) If $n$ is even, by a similar argument as in the proof of Lemma \ref{lem-LLP}, we know that $\lambda (T_{n,2,q-1})$ is the largest root of 
\begin{align*}
 g_1(x)&:= x^4 - x^3  - (n^2 x^2)/4 - q x^2 + x^2 + (n^2 x)/4 + (n x)/2  - n q x \\ 
&\quad -(n^2/4) + (n q)/2 + (n^2 q)/4 - (n q^2)/2   .  
\end{align*}
Let $f_1(x)$ be defined in Lemma \ref{lem-LLP}. 
Since $q=0.8 \sqrt{n}$, by computation, we see that 
\[ x f_1(x) - g_1(x) =  ( 0.8 \sqrt{n} -1) x^2 - 0.5 n x
  - 0.2 n^{5/2} + 0.57 n^2 - 0.4 n^{3/2} . \]
For every $n\ge 4$ and $x> {n}/{2}$, it follows that 
\[  \frac{\mathrm{d}}{\mathrm{d} x}(xf_1(x) - g_1(x)) =  (1.6 \sqrt{n} -2) x -0.5 n  >0. \] 
Then for every $x> n/2$, we have 
\begin{align*}
  x f_1(x) - g_1(x) \ge \tfrac{n}{2}f_1(\tfrac{n}{2}) - g_1(\tfrac{n}{2})  
  =0.07 n^2  -0.4 n^{3/2}  >0. 
  \end{align*}
 We conclude that $xf_1(x) > g_1(x)$ for every $x> n/2$ and $n\ge 33$. 
 Note that $\lambda (Y_{n,2,q})$ and $\lambda (T_{n,2,q-1})$ are the largest roots of $xf_1(x)$ and $g_1(x)$, respectively. 
Since $\lambda (Y_{n,2,q}) > \frac{n}{2}$ and $\lambda (T_{n,2,q-1})> \frac{n}{2}$, we obtain $\lambda (Y_{n,2,q}) < \lambda (T_{n,2,q-1})$. 

\medskip 
(b) If $n$ is odd, by computation, we get that $\lambda (T_{n,2,q-1})$ is the largest root of 
\begin{equation*}
\begin{aligned}
g_2(x) &:= x^4- x^3+ (5 x^2)/4 - (n^2 x^2)/4 - q x^2 - (
 3 x)/4 + (n x)/2 + (n^2 x)/4 + q x - n q x \\ 
& \quad  +1/4 - n^2/4 - (3 q)/4 + (n q)/2 + (n^2 q)/4 + q^2/2 - (n q^2)/2  .
\end{aligned}
\end{equation*}
Let $f_2(x)$ be determined in Lemma \ref{lem-LLP}. 
Substituting $q=0.8 \sqrt{n}$, we compute that 
\begin{equation*}
\begin{aligned}  
xf_2(x) -g_2(x) &= ( 0.8 \sqrt{n} -1) x^2 - 0.5(n-1) x  - 0.2 n^{5/2} + 0.57 n^2 \\ & \quad  - 0.4 n^{3/2} - 0.32 n + 0.6 \sqrt{n} -0.25 .
\end{aligned}
\end{equation*}
For every $x> \frac{n-1}{2}$, we have 
 \[ \frac{\mathrm{d}}{\mathrm{d} x} (xf_2(x) - g_2(x) ) = (1.6 \sqrt{n} -2)x - 0.5(n-1) >0. \]
 It follows that for every $n\ge 110$ and $x> \frac{n-1}{2}$, 
 \[ xf_2(x) - g_2(x) > \tfrac{n-1}{2} f_2(\tfrac{n-1}{2}) - g_2(\tfrac{n-1}{2}) 
 = 0.07 n^2 - 0.8 n^{3/2} + 0.68 n + 0.8 \sqrt{n} - 0.75 > 0. \] 
Thus, we get $xf_2(x) > g_2(x)$ for every $x> \frac{n-1}{2}$. 
It is easy to see that $\lambda (Y_{n,2,q})> \frac{n-1}{2}$ and $\lambda (T_{n,2,q-1}) > \frac{n-1}{2}$. So we conclude that $\lambda (Y_{n,2,q}) < \lambda (T_{n,2,q-1})$. 
\end{proof}

\subsection{Counting the number of triangles}

In 2007, Bollob\'{a}s and Nikiforov  \cite{BN2007jctb} 
established some powerful results on counting the number of 
triangles (cliques) of a graph 
in terms of the spectral radius.

\begin{lemma}[See \cite{BN2007jctb,CFTZ20,NZ2021}] 
\label{thm-BN-CFTZ-NZ}
Let $G$ be a  graph with $m$ edges. Then 
\begin{equation*} 
  t(G) \ge \frac{\lambda \bigl(\lambda^2 - m\bigr)}{3}. 
  \end{equation*}
    The equality holds if and only if $G$ is a complete bipartite graph.  
\end{lemma}

As a consequence, Lemma \ref{thm-BN-CFTZ-NZ} implies an interesting result of Nosal \cite{Nosal1970}, which states that every $m$-edge graph $G$ with  $\lambda (G)> \sqrt{m}$ contains a triangle; see \cite{Niki2002cpc,Ning2017-ars} for related results. 
The inequality in Lemma \ref{thm-BN-CFTZ-NZ} can be rephrased as the following versions: 
\begin{equation*} 
   \lambda^3 \le 3t + m \lambda 
\quad \Leftrightarrow \quad {t\ge \frac{1}{3}\lambda (\lambda^2 - m)  } 
\quad \Leftrightarrow \quad m \ge \lambda^2- \frac{3t}{\lambda}. 
\end{equation*}
This inequality was initially obtained by 
Bollob\'{a}s and Nikiforov \cite[Theorem 1]{BN2007jctb}, 
and it was independently discovered by Cioab\u{a},  Feng, 
Tait and Zhang \cite{CFTZ20}, as well as Ning and Zhai \cite{NZ2021}.  
The case of equality was characterized by Ning and Zhai \cite{NZ2021}.


\medskip 
Moon--Moser's inequality states that if 
$G$ has $n$ vertices and $m$ edges, 
then 
\[ t(G)\ge \frac{4m}{3n} \left( m - \frac{n^2}{4} \right). \]  
We refer to \cite[p. 297]{Bollobas78} and \cite[p. 443]{Lov1979} for detailed proofs. 
This result gives the supersaturation on triangles 
for graphs with more than $n^2/4$ edges. 
Next, we present a generalization for a graph with 
less than $n^2/4$ edges whenever this graph 
is far from being bipartite. 
More precisely, for $\varepsilon >0$, 
we say that a graph $G$ is {\it $\varepsilon$-far from being bipartite} if 
any subgraph $G'$ of $G$ 
with more than $e(G)-\varepsilon$ edges is not bipartite. 
In other words, 
if $G$ is $\varepsilon$-far from being bipartite, then we must  
delete at least $\varepsilon$ edges from $G$ to make it bipartite.

\begin{theorem}[See \cite{BBCLMS2017}] \label{thm-far-bipartite} 
If $G$ is $\varepsilon$-far from being bipartite, then 
\[ t(G) \ge \frac{n}{6}\left( m +\varepsilon-\frac{n^2}{4}\right).  \] 
\end{theorem}  

For the convenience of the reader, we include a proof here.  

\begin{proof} 
 For each vertex $v \in V(G)$, let $N_v=N(v)$ 
 and $N_v^c=V(G)\setminus N(v)$.  
Since $G$ is $t$-far from being bipartite, we must have $e(N_v) + e(N_v^c) \ge \varepsilon$ for every $v\in V(G)$. 
Indeed, if $e(N_v) + e(N_v^c) < \varepsilon$ for some $v$, then deleting all edges inside $N_v$ and $N_v^c$ yields a bipartite subgraph with more than $e(G) - \varepsilon$ edges, a contradiction.  
For each $v\in V(G)$, we have 
\[  \sum_{w\in N_v^c} d(w) = 2e(N_v^c) + e(N_v^c,N_v) 
=e(N_v^c) + m - e(N_v) 
\ge m+ \varepsilon -2e(N_v).  
  \]
Summing this inequality over all vertices $v\in V(G)$ gives  
  \[  \sum_{v\in V} \sum_{w\in N_v^c} d(w) 
  \ge mn + \varepsilon n - 2 \sum_{v\in V} e(N_v) = 
  mn + \varepsilon n - 6t(G), \]
  where we used the identity $3t(G) =\sum_{v\in V}e(N_v)$. 
 On the other hand, we get  
\[ \sum_{v\in V} \sum_{w\in N_v^c} d(w)= 
\sum_{v\in V} \left( 2m - \sum_{w\in N_v}d(w) \right) 
= 2mn - \sum_{w\in V} d^2(w). \]
Combining the two inequalities, 
we obtain 
\[  6t(G)  
\ge \varepsilon n - mn + \sum_{w\in V} d^2(w) 
\ge  \varepsilon n - mn + \frac{4m^2}{n}, \]
where the last inequality follows by Cauchy--Schwarz's  inequality. 
Note that  
${4m^2}/{n} \ge 2mn - {n^3}/{4}$. Substituting this into the previous inequality gives the desired bound. 
\end{proof}

The supersaturation on triangles is usually deduced from 
 Moon--Moser's inequality. 
In fact, Theorem \ref{thm-far-bipartite} 
can also imply the supersaturation on triangles.  
Note that every bipartite graph of order $n$ has at most $n^2/4$ edges. 
Thus,  if $G$ has at least ${n^2}/{4} +q$ edges, 
then it is $q$-far from being bipartite. 
 Theorem \ref{thm-far-bipartite} implies that 
$G$ contains at least ${qn}/{3}$ triangles.  
In particular, setting $q=\varepsilon n^2$, 
we obtain that 
$G$ has  at least ${\varepsilon n^3}/{3}$ triangles.  

\medskip 
For dealing with some extremal graph problems, 
it is a great advantage that 
applying Theorem \ref{thm-far-bipartite} can 
help us to get rid of the use of triangle removal lemma so that 
we do not require the order of a graph to be sufficiently large. 
This approach could be considered to be a key ingredient in our paper, and 
it was proven to be effective and feasible in many extremal graph problems. 
For other related applications, 
we refer the reader to \cite{LFP2024-triangular,LFP-count-bowtie}.

\subsection{A spectral stability for graphs with few triangles}

In what follows, we provide a warm-up practice    
that has its independent interest. 
We concentrate on explaining the main ideas which we 
believe have wide applicability.

\begin{theorem} \label{thm-tri-effi}
If $G$ is a graph on $n$ vertices with 
$\lambda (G)\ge n/2$ and $t(G)\le kn/2$, then 
\[ e(G)> \frac{n^2}{4} - 3k,  \] 
 and there exists a vertex partition 
 $V(G)=S\cup T$ such that 
 \[  e(S,T)\ge \frac{n^2}{4} -9k \] 
  and 
 \[ \frac{n}{2} - 3\sqrt{k} \le |S|, |T| \le  \frac{n}{2} + 3\sqrt{k}. \]  
 Moreover, we have 
 \[   \frac{n}{2} -12 k \le \delta (G)\le 
 \lambda (G)\le \Delta (G) \le  
  \frac{n}{2} + 9k . \] 
\end{theorem}

Theorem \ref{thm-tri-effi} 
gives an effective way to obtain a bipartition of a graph with large spectral radius and few triangles.  To some extend, 
this theorem indicates how the supersaturation-stability method 
can be used to deal with a variety of spectral extremal problems.  
The proof of Theorem \ref{thm-tri-effi}  is a well combination of 
Lemma \ref{thm-BN-CFTZ-NZ} and Theorem 
\ref{thm-far-bipartite}.

\begin{proof} 
 First, if $G$ is a complete bipartite graph, 
then $e(G)\le \lfloor n^2/4\rfloor$. 
Since  $e(G)=\lambda^2(G) \ge n^2/4$,  
it follows that $n$ is even and $G=K_{\frac{n}{2}, \frac{n}{2}}$. In this case, 
 the desired result holds immediately.  
  Now suppose that $G$ is not a complete bipartite graph.  
Then the inequality in Lemma \ref{thm-BN-CFTZ-NZ} is strict.  So we get $e(G)> \lambda^2 - (3t)/ \lambda 
 \ge \lambda^2 - (6t)/n \ge n^2/4 -3k$. 
 Theorem \ref{thm-far-bipartite} implies that $G$ is not 
 $6k$-far from being bipartite. Otherwise, if $G$ is 
 $6k$-far from being bipartite, then $t(G)
 \ge {n}/{6}\cdot (e(G) + 6k - n^2/4) >   kn /2$, 
 which contradicts the assumption. 
 Hence $G$ is not $6k$-far from being bipartite. 
Consequently, there exists a partition $V(G)=S\cup T$ such that 
 $e(S) + e(T)< 6k$. Moreover, we get 
 $e(S,T)= e(G) - e(S) - e(T) > n^2/4 - 9k$. 
It follows that 
  $n/2 - 3\sqrt{k} < |S|, |T| < n/2 + 3\sqrt{k}$. 
Note that $e(K_{|S|,|T|}) - e(S,T) < |S| |T| - (n^2/4 - 9k) \le 9k$. 
In other words,  there are fewer than $9k$ missing edges between $S$ and $T$. Consequently, we get $\delta (G)>  
(n/2 - 3\sqrt{k}) - 9k \ge n/2- 12k$. 
Finally, since $e(S) + e(T)< 6k$, we have  
$\Delta (G)< (n/2 + 3\sqrt{k}) +6k \le n/2 +9k$, as desired. 
\end{proof}

\section{Proof of Theorem \ref{thm-Yn2q}}
\label{sec4}

The case $q=1$ was proved by Li, Lu and Peng \cite{LLP2024-AAM}. 
Next, we consider the case $2\le q \le \frac{1}{11}\sqrt{n}$
although our proof works for $q=1$. 
Throughout this section, 
let $G$ be a graph on $n$ vertices with
 $\lambda (G)\ge \lambda (Y_{n,2,q})$. 
Suppose on the contrary that $t(G)< q\lfloor {n}/{2}\rfloor$. 
In the sequel, we will deduce a contradiction. 
The key insight of our proof lies in the application of the supersaturation-stability (Theorem \ref{thm-far-bipartite}). 
We outline the main steps as follows. 

\begin{itemize}
\item[\ding{73}] 
Using Lemma \ref{thm-BN-CFTZ-NZ}, 
we get that $G$ contains more than $\lfloor n^2/4\rfloor - q$ edges. 
Theorem \ref{thm-far-bipartite} 
implies that we can delete less than 
$4q$ edges from $G$ to make it bipartite. 
So we get a bipartition $V(G) =S\cup T$ with 
$e(S) + e(T)< 4q$ and $ |S|,|T| \approx \frac{n}{2} \pm \sqrt{5q}$.

\item[\ding{73}] 
Secondly, we show that 
$e(S) + e(T)\le q$ and  $\delta (G) \ge \frac{n}{2} - 2q$; see Lemmas \ref{lem-ST-q} and \ref{refine-degree}. 
Let $\bm{x}\in \mathbb{R}^n$ be the Perron vector of $G$ satisfying $\max\{x_v: v\in V(G)\}=1$. 
Then we will prove in Lemma \ref{eigen-entry} 
that ${x}_u \ge 1- {20q}/{n}$ 
for every $u\in V(G)$.

\item[\ding{73}] 
We prove in Lemma \ref{balanced} 
that the bipartition of $G$ 
is balanced, i.e., $\bigl| |S| - |T|\bigr| \le 1$. 
Then we shall prove that 
$t(G)\ge q\lfloor n/2\rfloor$, 
or $G$ satisfies $\lambda (G) < \lambda (Y_{n,2,q})$, a contridiction. 

\end{itemize}

In what follows, we divide the proof into a sequence of lemmas.

 \begin{lemma} \label{approx-partition} 
 There is a partition $V(G)=S\cup T$ such that 
 \[  e(S) + e(T)< 4q  \]
 and 
 \[ e(S,T)> \frac{n^2}{4} - 5q. \]
Consequently, we get 
 \[ \frac{n}{2} - \sqrt{5q} < |S|, |T| < \frac{n}{2} + \sqrt{5q}.  \]
 \end{lemma}
 
 \begin{proof}
By Lemma \ref{lem-LLP}, 
 we know  that $\lambda^2 (G) \ge \lambda^2 (Y_{n,2,q})> \lfloor n^2/4\rfloor +2q$. 
 Since $t(G)< q\lfloor {n}/{2}\rfloor$ and $\lambda (G)> \frac{n}{2}$, 
  by  Lemma \ref{thm-BN-CFTZ-NZ}, we have 
\begin{equation} \label{eq-eG}   
e(G)\ge \lambda^2(G) - \frac{3t(G)}{\lambda (G)} 
>  \lambda^2(G) - \frac{6t(G)}{n} >    
\left\lfloor \frac{n^2}{4} \right\rfloor - q . 
\end{equation}
We claim that $G$ is not $4q$-far from being bipartite. 
Otherwise, if $G$ is $4q$-far from being bipartite, 
then using Theorem \ref{thm-far-bipartite}, we get 
\[  t(G)\ge \frac{n}{6} \left(e(G) + 4q - \frac{n^2}{4} 
\right)> q \left\lfloor \frac{n}{2} \right\rfloor, \]  
which contradicts with the assumption.  
So $G$ is not $4q$-far from being bipartite. 
That is to say, we can delete less than $4q$ edges from 
$G$ to make it bipartite, i.e., 
there exists a vertex partition 
$V(G)=S\cup T$ such that 
\begin{equation*}
e(S) + e(T)< 4q. 
\end{equation*}
From the above discussion, we get 
\[  e(S,T) 
= e(G) - e(S) - e(T) >  \frac{n^2}{4} - 5q. \] 
Without loss of generality, we may assume that 
$|S| \le |T|$. 
Suppose on the contrary that $|S| \le \frac{n}{2} - \sqrt{5q}$. 
Then $|T|= n-|S| \ge \frac{n}{2} + \sqrt{5q}$. It yields that 
\[   e(S,T)\le |S| |T| \le \left(\frac{n}{2} - \sqrt{5q} \right) 
\left(\frac{n}{2} +  \sqrt{5q} \right) = \frac{n^2}{4} -5q, \] 
 a contradiction. 
Thus, we must have 
 $\frac{n}{2} - \sqrt{5q} < |S|, |T| < \frac{n}{2} +\sqrt{5q}$.  
\end{proof}

\begin{lemma} \label{lem-ST-q}
We have $e(S) + e(T)\le q$. 
\end{lemma}

\begin{proof}
Suppose on the contrary that 
$e(S) + e(T) = k$, where $k$ is an integer with $q+1\le k < 4q$ by Lemma \ref{approx-partition}.  
For simplicity, we write $K_{S,T}$ for the complete bipartite graph 
on the vertex parts $S$ and $T$. 
Note that every triangle of $G$ must contain an edge from $G[S]\cup G[T]$,  and each edge of $G[S]\cup G[T]$ yields at least $\min\{|S|,|T|\}$ triangles.  
Each missing edge between 
$S$ and $T$ destroys at most $k$ triangles. 
By Lemma \ref{approx-partition}, we have 
$\min\{|S|,|T|\} > \frac{n}{2} -\sqrt{5q}$, and 
$G[S,T]$ misses at most $5q$ edges from $K_{S,T}$. 
We conclude 
\[  t(G) > k \left(\frac{n}{2} -\sqrt{5q} \right) - 5q\cdot k
 \ge (q+1)\left( \frac{n}{2} - \sqrt{5q} -5q \right) \ge q \left\lfloor \frac{n}{2} \right\rfloor, \]   
 where the last inequality holds for every $n\ge 20q^2$. 
This contradicts with the assumption $t(G)< q \lfloor n/2\rfloor$. 
Thus, we have $e(S) + e(T)\le q$.  
\end{proof}

\begin{lemma} \label{refine-degree}
With the previous notation, we have  
\[ e(S,T) > \frac{n^2}{4}  - 2q   \]
and 
 \[ \frac{n}{2} -  \sqrt{2q} < |S|, |T| 
 < \frac{n}{2} +\sqrt{2q}.  \]
Furthermore, it follows that  
\[  \frac{n}{2} -2q \le \delta(G) \le 
\lambda (G)\le \Delta (G) \le \frac{n}{2} + q + \sqrt{2q}.  \]
\end{lemma}

\begin{proof}
Combining (\ref{eq-eG}) with Lemma \ref{lem-ST-q},  we obtain 
$e(S,T) =e(G) - e(S) -e(T) > \lfloor n^2/4\rfloor - 2q$. 
By the AM-GM inequality, we can get 
$\frac{n}{2} - \sqrt{2q} \le |S|, |T| \le \frac{n}{2} +\sqrt{2q}$. 

By the Rayleigh formula, it is well-known that 
$\delta(G) \le 
\lambda (G)\le \Delta (G) $. 
Since $e(S) + e(T)\le q$,  
each vertex of $S$ has degree 
at most $q$ in $G[S]$, and this also holds for the vertices of $T$. 
Then it follows that $\Delta (G) \le 
\max\{|S|,|T|\} +q \le \frac{n}{2} + \sqrt{2q} +q$. 
Next, we show that $\delta (G)\ge \frac{n}{2} -2q$. 
Otherwise, if there exists a vertex $v\in V(G)$ with degree and $d(v) \le \frac{n}{2} - 2q  - \frac{1}{2}$, then using (\ref{eq-eG}) gives 
\begin{eqnarray*} 
e(G \setminus \{v\})  = e(G) - d(v) 
&\ge & \left( \Big\lfloor  \frac{n^2}{4} \Big\rfloor   -q +1 \right) 
- \left( \frac{n}{2} -2q - \frac{1}{2} \right)  \\
& \ge &\frac{(n-1)^2}{4} +q +1.
\end{eqnarray*}   
By Theorem \ref{thm-LS1975}, we can find at least 
$(q+1)\lfloor \frac{n-1}{2} \rfloor$ triangles in 
$G \setminus \{v\}$. This is a contradiction with 
the assumption $t(G)< q\lfloor \frac{n}{2}\rfloor$. Thus, we must have $\delta (G)\ge \frac{n}{2} -2q $. 
\end{proof}

Recall that $\bm{x}\in \mathbb{R}^n$ is the Perron vector of $G$ with $\max\{x_v: v\in V(G)\}=1$.

\begin{lemma} \label{eigen-entry}
For every $u\in V(G)$, we have 
\[  {x}_u > 1- \frac{20q}{n}. \]  
\end{lemma}

\begin{proof} 
Let $z\in V(G)$ be a vertex 
such that ${x}_z$ attains the maximum 
eigenvector entry of $\bm{x}$ and ${x}_z=1$. 
Without loss of generality, we may assume that $z\in S$.  
Our proof of this lemma falls naturally into two steps.  
\medskip 

{\bf Step 1.}   For every $u \in S$, 
we have ${x}_u > 1- \frac{11q}{n}$. 

\medskip 

 Using Lemma \ref{lem-ST-q} gives $e(S)\le q$ and  $d_S(u)\le q $. By Lemma \ref{refine-degree},  we have  
$d_T(u) = d(u)-d_S(u)  \ge ( \frac{n}{2}-2q )-q  = \frac{n}{2}-3q$. 
By Lemma \ref{refine-degree} again, we  get 
$|T|\le \frac{n}{2} + \sqrt{2q}$. Then 
\begin{eqnarray*}
 |N_T(z)| - |N_T(u)\cap N_T(z)|  
 &=& -|N_T(u)| + |N_T(u) \cup N_T(z)|  \\
& \le& - \left(\frac{n}{2} -3q \right) +  
\left(\frac{n}{2}  + \sqrt{2q} \right) \\
&=& 3q + \sqrt{2q} .
\end{eqnarray*}  
Hence, we have 
\begin{eqnarray*}
\lambda {x}_u-\lambda {x}_z 
= \sum_{v\sim u} {x}_v  - \sum_{v\sim z} {x}_v 
\ge  -\sum_{v\in T, v\sim z, v\not\sim u} {x}_v 
-\sum_{v\in S, v\sim z} {x}_v. 
\end{eqnarray*}
Applying the fact ${x}_v \le 1$, it follows that 
\begin{eqnarray*}
\lambda {x}_u-\lambda {x}_z  
\ge  -\Bigl( |N_T(z)| -|N_T(u)\cap N_T(z)| \Bigr) 
- |N_S(z)| 
\ge  - 4q -\sqrt{2q}. 
\end{eqnarray*}
Recall that ${x}_z=1$ 
and $\lambda > {n}/{2}$. Then for each $u\in S$, 
we have 
\begin{equation*}\label{verct1}
{x}_u\ge 1-\frac{4q + \sqrt{2q} }{\lambda }> 1-\frac{8q + 2\sqrt{2q} }{n} > 1-\frac{11q}{n}.
\end{equation*}

{\bf Step 2.}   
For every $u\in T$, we have ${x}_u> 1- \frac{20q}{n}$. 

\medskip 

By the previous step, we get 
 $$\lambda  {x}_u=\sum_{v\in N(u)} {x}_v\ge \sum_{v\in  N_S(u)} {x}_v 
 > \left(1-\frac{8q + 2\sqrt{2q} }{n} \right)d_S(u).$$ 
By  Lemma \ref{refine-degree}, we can see that 
$d(u)\ge \delta (G)\ge \frac{n}{2} - 2q$, 
which together with $d_T(u)\le q$ gives 
$ d_S(u)=d(u)- d_T(u)\ge  \frac{n}{2}-3q$.
 Hence, for $q\ge 2$, we have 
 \begin{eqnarray*}
 {x}_u >  \frac{(1-\frac{8q + 2\sqrt{2q} }{n})d_S(u)}{\lambda } 
 \ge  \frac{(1-\frac{8q + 2\sqrt{2q} }{n})(\frac{n}{2}-3q)}{
 \frac{n}{2}+q + \sqrt{2q}} 
 \ge 1- \frac{20q}{n}.
 \end{eqnarray*} 
Combining the above two claims, 
 the desired result follows immediately.
\end{proof}

Lemma \ref{eigen-entry} gives 
${x}_u > 1- \frac{20q}{n}$, 
we will show that the bipartition $V(G)=S\cup T$ is balanced 
(see Lemma \ref{balanced}). 
To start with, we fix some notation. 
Let $K_{S, T}$ be the complete bipartite graph with partite sets $S$ and $T$. Let $G_1 = G[S] \cup G[T]$ be 
the graph whose edges consist of the class-edges, 
and let $G_2$ be the graph on $V(G)$ with the missing edges between $S$ and $T$, that is,  $E(G_2)=E(K_{S,T}) \setminus E(G)$. Note that $e(G) = 
e(G_1) + e(K_{S,T})  - e(G_2)$.

\begin{lemma} \label{over-square}
 We have 
\[  \frac{2}{n} \left \lfloor\frac{n^2}{4} \right \rfloor 
 - \sqrt{|S||T|}   < \frac{80q^2}{n(n- 40q)}. \]
\end{lemma}

\begin{proof} 
 By Lemma~\ref{lem-ST-q}, we get 
 $e(G_1) \leq q$. 
By Lemma \ref{eigen-entry}, we have 
\begin{equation*} \label{eqxx}
\bm{x}^{\mathrm{T}} \bm{x} >  
n \left(1-\frac{20q}{n} \right)^2 
> n \left(1-\frac{40q}{n} \right)=n- 40q.
\end{equation*} 
Using Rayleigh's formula, it follows  that 
 \begin{eqnarray*} 
 \lambda (G)
 &=& \frac{\bm{x}^{\mathrm{T}} \bigl( A(K_{S,T}) + 
 A(G_1) - A(G_2) \bigr)\bm{x}}{\bm{x}^{\mathrm{T}}\bm{x}} \\
  &=&  \frac{\bm{x}^{\mathrm{T}} A(K_{S,T})\bm{x}}{\bm{x}^{\mathrm{T}}\bm{x}}+ \frac{\bm{x}^{\mathrm{T}} A(G_1)\bm{x}}{\bm{x}^{\mathrm{T}}\bm{x}} \\
  &< &
   \sqrt{|S||T|} + \frac{2 q}{n-40q} .
 \end{eqnarray*}
 On the other hand,  it is familiar that 
 \[   \lambda (G)\ge \lambda (Y_{n,2,q}) \ge \frac{2e(Y_{n,2,q})}{n} 
 \ge  \frac{2}{n} \left \lfloor\frac{n^2}{4} \right \rfloor
 + \frac{ 2 q}{n}. \]
 Then we have 
 \begin{eqnarray*}
 \frac{2}{n} \left \lfloor\frac{n^2}{4} \right \rfloor 
  - \sqrt{|S||T|}  
<
    2q \left(\frac{1}{n- 40q }-\frac{1}{n}  \right) 
   = \frac{80q^2}{n(n-40q)}. 
 \end{eqnarray*}
 This completes the proof. 
 \end{proof}

\begin{lemma} \label{balanced}
We have $\bigl| |S| - |T| \bigr|\le 1$. 
\end{lemma}

 \begin{proof}
  We assume on the contrary that $|T|\ge |S|+2$.
Next, considering the following two cases, 
and we will derive a contradiction in each case. 

\medskip 
 {\bf Case 1.} $n$ is even. Since $|S|+|T|=n$, we get
 \begin{eqnarray*}
  \frac{2}{n} \left \lfloor\frac{n^2}{4} \right \rfloor  
    -\sqrt{|S||T|} & \ge&  \frac{n}{2}-\sqrt{ \left (\frac{n}{2}-1 \right) \left(\frac{n}{2}+1 \right)}\\
&=&\frac{n}{2}-\sqrt{\frac{n^2}{4}-1}=\frac{1}{\frac{n}{2}+\sqrt{\frac{n^2}{4}-1}}>\frac{1}{n}.
  \end{eqnarray*}

 {\bf Case 2.} $n$ is odd. Since $|S|+|T|=n$, we  have
   \begin{eqnarray*}
    \frac{2}{n}\left \lfloor\frac{n^2}{4} \right \rfloor  
     -\sqrt{|S||T|}
    &\ge & \frac{n^2-1}{2n}-\sqrt{\left (\frac{n-3}{2} \right) \left(\frac{n+3}{2} \right)}\\
&=& \frac{1}{2}
\left( {n-\frac{1}{n}}-{\sqrt{n^2-9}}\right)  
= \frac{7+\frac{1}{n^2}}{2(n-\frac{1}{n}+\sqrt{n^2-9})}> \frac{1}{n}.
  \end{eqnarray*}
By Lemma \ref{over-square}, we get 
 $$\frac{1}{n}< \frac{2}{n} \left \lfloor\frac{n^2}{4} \right \rfloor   
  -\sqrt{|S||T|} <  \frac{80q^2}{n(n-40 q)} .
 $$
This leads to a contradiction for $q\ge 2 $ and $n\ge 100q^2$. 
\end{proof}

Recall that $G$ is an $n$-vertex graph 
with less than $ q \lfloor {n}/{2}\rfloor$ triangles. 
In what follows, 
we will deduce a contradiction in two cases. 
On the one hand, 
if  $G$ has at least $ \lfloor n^2/4\rfloor +q$ edges, then 
 Theorem \ref{thm-LS1975} implies that 
$G$ contains at least $q \lfloor {n}/{2}\rfloor$ triangles, 
which is a contradiction with the assumption. 
On the other hand, 
if $G$ has at most $  \lfloor n^2/4\rfloor +q -1$ edges, 
then we will prove that $\lambda (G) < \lambda (Y_{n,2,q})$, 
which leads to a contradiction as well. 
This will completes the proof of Theorem \ref{thm-Yn2q}. 
We write the second case as the following lemma.

\begin{lemma} \label{lessYn2q}
If $e(G)\le  \lfloor n^2/4\rfloor +q -1$, then 
$\lambda (G)< \lambda (Y_{n,2,q})$. 
\end{lemma}

\begin{proof}
By Lemma \ref{balanced}, we know 
that $\bigl| |S|-|T| \bigr|\le 1$. 
Recall that $Y_{n,2,q}$ is the graph on the vertex sets 
 $S$ and $T$ with all edges between $S$ and $T$, together with 
$q$ pairwise disjoint edges embedded 
into the larger partite set. 
 Clearly, we have $e(Y_{n,2,q})=  \lfloor n^2/4\rfloor +q$. 
We denote 
\[  E_{+}:=E(Y_{n,2,q})\setminus E(G) \] 
 and 
 \[ E_{-} :=E(G)\setminus E(Y_{n,2,q}). \]  
 Note that $(E(G) \cup E_+) \setminus E_- = E(Y_{n,2,q})$. 
 It follows that 
 \[  e(G)+|E_+| - |E_-|=e(Y_{n,2,q}). \]  
 Using the assumption 
 $  e(G) \le  \lfloor n^2/4\rfloor +q -1$,  we get 
    \[   |E_+| \geq |E_-| + 1.\] 
    Furthermore, we know from Lemma \ref{lem-ST-q} that 
    $|E_-| \leq e(S) + e(T) \le q$. 
    By  Lemma \ref{refine-degree}, 
    we have that $|E_+| \le \lfloor {n^2}/{4} 
    \rfloor- e(S,T) + q \le 3q$. 
    By  Lemma~\ref{eigen-entry}, we know that 
    ${x}_u > 1- {(20q)}/{n}$ for every $u\in V(G)$. Therefore, we obtain 
\begin{eqnarray*}
\lambda (Y_{n,2,q}) &\geq & 
\frac{2}{\bm{x}^{\mathrm{T}} \bm{x}}
 \sum_{\{i,j\}\in E(Y_{n,2,q})} {x}_i {x}_j  \\
&= & \frac{2}{\bm{x}^{\mathrm{T}} \bm{x}} \sum_{\{i,j\} \in E(G)}  
{x}_i {x}_j  
+\frac{2}{\bm{x}^{\mathrm{T}} \bm{x}} \sum_{\{i,j\} \in E_+}  
{x}_i {x}_j - \frac{2}{\bm{x}^{\mathrm{T}} \bm{x}} \sum_{\{i,j\}\in E_-}  {x}_i {x}_j \\
& >  &
\lambda (G) + \frac{2}{\bm{x}^{\mathrm{T}} \bm{x}} 
\left( |E_+| \Bigl(1 - \frac{20q}{n}\Bigr)^2 - |E_-|\right)\\
& =& \lambda (G) + \frac{2}{\bm{x}^{\mathrm{T}} \bm{x}} \left( |E_+|  -\frac{40q}{n}  |E_+| + 
\frac{(20q)^2}{n^2}  |E_+| - |E_-| \right)\\
& \geq& \lambda (G) + \frac{2}{\bm{x}^{\mathrm{T}} \bm{x}} \left( 1-\frac{40q}{n}  |E_+| +\frac{(20q)^2}{n^2}  |E_+| \right)\\
&>& \lambda (G) , 
\end{eqnarray*}
where the last inequality holds since $n\ge 120q^2$ and $|E_+| \le 3q$. 
\end{proof}

From the above proof of Lemma \ref{lessYn2q}, 
we can immediately obtain the following proposition, which can be viewed as a supplement of Example \ref{exam}.

\begin{proposition}
    If $q\le \frac{1}{11}\sqrt{n}$ and $G$ is a graph obtained from $T_{n,2}$ 
    by adding $q-1$ edges into the color classes arbitrarily, 
    then $t(G)< q \lfloor {n}/{2}\rfloor$ and 
$ \lambda (G) < \lambda (Y_{n,2,q})$. 
\end{proposition}

\section{Proof of Theorem  \ref{thm-sp-LS}}

\label{sec5}

In this section, we are ready to prove Theorem \ref{thm-sp-LS}. Equivalently,  Theorem \ref{thm-sp-LS} states that if $G$ is an $n$-vertex graph with at most $ q \lfloor {n}/{2}\rfloor$ triangles, then $\lambda (G)\le \lambda (T_{n,2,q})$, with equality if and only if $G=T_{n,2,q}$.   
Without loss of generality, we may assume that $G$ attains the maximum spectral radius.  Our goal is to prove that $G=T_{n,2,q}$.  
The maximality of $G$ implies $ \lambda (G) > \lambda (Y_{n,2,q}) $. So Theorem \ref{thm-Yn2q} is applicable, and  $G$ has exactly $q\lfloor {n}/{2}\rfloor$ 
triangles.
Similarly, the previous Lemma \ref{approx-partition} -- 
 Lemma \ref{balanced} still hold in this setting. 
By Lemma \ref{lessYn2q}, it follows that $e(G)= \lfloor n^2/4\rfloor +q$, and $G$ is obtained from $T_{n,2}$ by embedding 
exactly $q $ edges into the color classes, 
and there is no triangle among these $q$ edges.  
Combining these facts, 
it remains to show that these $q$ edges induce a star in the vertex part of size 
$\lceil {n}/{2}\rceil$. 

\medskip 
Observe that $G$ does not contain 
$q +1$  pairwise edge-disjoint triangles. 
In particular, if  $n$ is even and $q\neq 3$, 
then $|S|=|T|={n}/{2}$ and setting $k:=q+1$ in \cite[Theorem 7]{LZZ2022}, 
we know that $\lambda (G)\le \lambda (T_{n,2,q})$, 
with equality if and only if $G=T_{n,2,q}$.  
However, their proof in \cite{LZZ2022}  
follows the line of that of forbidding friendship graph in \cite{CFTZ20}, 
where the triangle removal lemma is used  
so it requires that $q$ is fixed and 
$n$ is sufficiently large. 
Actually,  the result seems ambiguous whenever $q$ is large 
(say, when $q\ge \log n$).   
The authors \cite{LFP2024-triangular} 
recently provided an alternative new method that 
avoids the use of triangle removal lemma. 

\medskip 
From the above discussion, 
one of the main difficulties of proving Theorem \ref{thm-sp-LS}  lies in 
giving a unique characterization of the spectral extremal graph. 
In the sequel, we shall  give a complete 
proof for every $n\ge 121q^2$, instead of for sufficiently large $n$.  
The key ingredient relies on the double-eigenvector technique, 
which was initially used by Rowlinson \cite{Row1988}.  
As far as we know, this technique turns out to be effective 
for treating spectral extremal graph problems; see \cite{ZLX2022,LZZ2022,FTZ2024,LLZ2024,ZFL2024} and references therein. 
The idea in our proof has its root in \cite{LZZ2022,FTZ2024}. 
To begin with, we need to introduce two lemmas. 
The first lemma  provides an operation of a connected graph that strictly increases the adjacency spectral radius.

 \begin{lemma}[See \cite{WXH2005}]  \label{lem-WXH}
 Assume that $G$ is a connected graph with 
 $u,v\in V$ and $w_1,\ldots ,w_s\in N(v)\setminus N(u)$. 
 Let $\bm{x}=(x_1,\ldots ,x_n)^{\mathrm{T}}$ be the Perron vector 
 with $x_v$ corresponding to the vertex $v$. 
 Let $G'=G-\{vw_i: 1\le i\le s\} + 
 \{uw_i: 1\le i \le s\}$. If ${x}_u\ge {x}_v$, 
 then $\lambda (G') > \lambda (G)$. 
 \end{lemma}

The second lemma gives an upper bound 
 on the sum of squares of degrees in a graph with given size. 
 Here, we provide a two-line simple proof for convenience. 
  
 \begin{lemma}[See \cite{Das2003,NZ2021b}] \label{lem-deg-squ}
If $G$ is a graph with $m$ edges, then  
\[  \sum_{v\in V(G)}d^2(v) \le m^2 +m. \]  
 \end{lemma} 
 
 \begin{proof}
For each $\{u,v\} \in E(G)$, we have 
$d(u) + d(v)\le m+1$. 
Then 
\[ \sum_{v\in V(G)} d^2(v) 
= \sum_{\{u,v\} \in E(G)} (d(u) + d(v)) 
\le m(m+1). \qedhere \]  
 \end{proof}

Recall that 
we have shown that $G$ is obtained from 
 $T_{n,2}$ by embedding $q $ edges into its partite sets. 
Without loss of generality, we may assume that $|S|\le |T|$. 
Then $|S|=\lfloor {n}/{2}\rfloor$ 
and $|T|=\lceil {n}/{2}\rceil$. 
For notational convenience, 
we denote $x_S:=\sum_{u\in S} x_u$ and $x_T:=\sum_{u\in T} x_u$. Let $H_1$ and $H_2$ be the subgraphs that are embedded 
into the vertex parts $S$ and $T$, respectively. 
To maximize $\lambda (G)$, 
we need to prove that $H_1=\varnothing$ 
and $H_2=K_{1,q}$.

\begin{lemma} \label{lem-star}
For each $i=1,2$, if $H_i \neq \varnothing$, then $H_i $ is a star 
or a $4$-cycle. 
\end{lemma}

\begin{proof}
Without loss of generality, 
we shall show that if $H_1\neq \varnothing$, 
then $H_1$ is a star. 
We choose $v^* \in S$ as a vertex such that 
${x}_{v^*} = \max\{ {x}_v: v\in S\}$. 
Clearly, the vertex $v^*$ can not be an isolated vertex of $G[S]$.  
First of all, we consider the case $1\le e(H_1)\le 3$. 
By enumerating, we have $H_1\in \{K_2, 2K_2, K_{1,2}, 3K_2, P_3\cup K_2, P_4,K_{1,3}\}$. 
We claim that $v^*$ is a dominating vertex of $H_1$. 
Otherwise, if there is a vertex $u \in V(H_1)\setminus N(v^*)$ and an edge $\{u,v\}\in E(H_1)$, 
then let $G'$ be 
the graph obtained from $G$ by deleting the edge 
$\{u,v\}$, and adding the edge $\{u,v^*\}$. 
Observe that $G'$ has at most $q \lfloor {n}/{2} \rfloor$ 
triangles. By Lemma \ref{lem-WXH}, 
we get $\lambda (G') > \lambda (G)$, a contradiction.  
Thus $v^*$ is a dominating vertex.
So $H_1$ is a star. 

Now, we consider the case $e(H_1)\ge 4$. 
Similar to the above argument,  
we can move the edges incident to $v^*$. 
Hence we can obtain that $d_{H_1}(v^*)\ge 3$, unless $H_1\cong C_4$. 
For the former case, we need to show that $H_1$ is a star.  
We denote $a= d_{H_1}(v^*)\ge 3$ and $H_1'=H_1\setminus  \{v^*\}$. 
If $e(H_1')=0$, then $H_1$ is a star, so we are done. 
If $b:=e(H_1')\ge 1$, then we define $G'$ as the graph 
obtained from $G$ by deleting 
all $b$ edges in $H_1'$, and then adding 
$b$ new edges from $v^*$ to $w_1,\ldots ,w_b 
\in S\setminus V(H_1)$. 
We can see that the edges in $G'[S]$ form a star $K_{1,a+b} $, 
 and $G^{\prime} $ 
contains at most $q \lfloor {n}/{2} \rfloor$ triangles. 
 Thus, we have $\lambda (G^{\prime} ) \leq \lambda (G)$. 
 We denote $\lambda = \lambda (G)$ and 
 $\lambda^{\prime}=\lambda\left(G^{\prime}\right)$. 
 Let $\bm{y}= (y_1, y_2, \ldots, y_n )^{\mathrm{T}}$ be the Perron vector of $G^{\prime}$.
By the symmetry, we have $y_v=y_{w_1}$ for every vertex $v \in V(K_{1,a+b}) \backslash \{v^*\}$. Thus, we have 
$\lambda^{\prime} y_{v^*}= y_T +(a+b) y_{w_1}$ and $\lambda^{\prime} y_{w_1}=y_T +y_{v^*}$, where 
$y_T=\sum_{u\in T} y_u$. Then 
\[  y_{v^*}=\frac{\lambda^{\prime}+a+b}{\lambda^{\prime}+1} y_{w_1} \geq \frac{\lambda+a+b}{\lambda+1} y_{w_1}. \] 
Using the double-eigenvector technique, we have
\begin{eqnarray}
\bm{x}^{\mathrm{T}} \bm{y}\left(\lambda^{\prime}-\lambda\right)  
&=& \bm{x}^{\mathrm{T}} \left( A(G') - A(G) \right)\bm{y} \notag \\[3mm]
&=& \sum_{\{u,v\}\in E(G')} (x_uy_v + x_vy_u) 
-  \sum_{\{u,v\}\in E(G)} (x_uy_v + x_vy_u) \notag \\ 
& = &\sum_{j=1}^b \left(x_{v^*} y_{w_j}+x_{w_j} y_{v^*}\right)-\sum_{\{u, v\} \in E(H_1')}\left(x_u y_v+x_v y_u\right) \notag \\ 
& \geq & b\cdot x_{v^*}y_{w_1} + b \cdot \frac{\lambda+a+b}{\lambda+1} x_{w_1}y_{w_1}-\sum_{\{u,v\} \in E(H_1')}\left(x_u+x_v\right) y_{w_1} , \label{eq3}
\end{eqnarray}
where we used the fact 
$x_{w_1}=\cdots =x_{w_b}$ and $y_v=y_{w_1}=\cdots =y_{w_b}$ for any $v\in V(H_1')$.

Next, we are going to prove 
\[   \bm{x}^{\mathrm{T}} \bm{y}(\lambda '- \lambda )\cdot  \lambda >0. \]  
Since $N_G(v^*) = T\cup N_{H_1}(v^*)$, 
we have  
\begin{equation}  \label{eq4}
\lambda x_{v^*} \geq x_{T}+a x_{w_1}. 
\end{equation}
 Moreover, as $N_{G}(w_1)=T$ and $\lambda x_{w_1}=x_{T}$, 
we have 
\begin{equation}  \label{eq5}
\frac{\lambda+a+b}{\lambda+1} \lambda x_{w_1} \geq 
\left( \lambda+a+b- \frac{3}{2} \right) x_{w_1}= 
x_{T}+ \left(a+b- \frac{3}{2} \right) x_{w_1}. 
\end{equation} 
Using the double counting, we get 
 \[  \sum_{\{u, v\} \in E(H_1')} \lambda \left(x_u+x_v\right) 
 = \sum_{v \in V(H_1')} 
 d_{H_1'}(v) \lambda x_v. \]  
For each $v \in V(H_1')$, we have $\lambda x_v 
\leq x_{T}+ (d_{H_1'}(v)+1 ) x_{v^*}$. 
By Lemma \ref{lem-deg-squ}, 
it follows that $\sum_{v \in V(H_1')} d_{H_1'}^2(v) \leq b^2+b$. Consequently, we get 
\begin{eqnarray}
\sum_{\{u, v\} \in E(H_1')} \lambda (x_u+x_v ) 
& \leq & \sum_{v \in V(H_1')} d_{H_1'}(v) x_T + 
\sum_{v \in V(H_1')} \left( d_{H_1'}^2 (v) 
+d_{H_1'}(v) \right) x_{v^*} \notag \\
& \leq & 2 b x_T + (b^2+3 b ) x_{v^*} . \label{eq6}
\end{eqnarray} 
Combining the above inequalities, we can obtain 
\begin{eqnarray*} 
\bm{x}^{\mathrm{T}} \bm{y} 
(\lambda '-\lambda )\cdot  \lambda & \overset{(\ref{eq3})}{\geq} & 
b y_{w_1}\cdot \lambda x_{v^*}
 + b y_{w_1} \cdot \frac{\lambda+a+b}{\lambda+1}  \lambda x_{w_1} -\sum_{w \in E(H_1')} \lambda (x_u+x_v) y_{w_1} \\
&\overset{(\ref{eq4},\ref{eq5},\ref{eq6})}{\ge} & by_{w_1} \left(2x_T + (2a+b-\tfrac{3}{2} )x_{w_1} \right)  
- (2bx_T  + (b^2+3b)x_{v^*} ) y_{w_1} \\[3mm]
& = & b(2 a+b- \tfrac{3}{2}) x_{w_1} y_{w_1} 
- b(b+3 ) x_{v^*} y_{w_1}.
\end{eqnarray*}
To obtain a contradiction, we would like to show that $\bm{x}^{\mathrm{T}} \bm{y} 
(\lambda '-\lambda )\cdot  \lambda>0$. 
It is sufficient to prove that 
$ \frac{x_{v^*}}{x_{w_1}} 
< \frac{2a+b - {3}/{2}}{b+3}  $.
Note that $\lambda x_{v^*} = x_T + \sum_{u\in N_{H_1}(v^*)} x_u \le x_T + a x_{v^*}$. Then 
$  x_{v^*} \leq \frac{x_{T}}{\lambda- a}$.
Observe that $x_{w_1}=\frac{x_T}{\lambda}$. 
Therefore, it yields that 
\[ \frac{x_{v^*}}{ x_{w_1} } \le 
\frac{\lambda}{ \lambda - a} . \]  
Since $a \geq 3$, we get $(2 a+b- \frac{3}{2} )
- (b+3 )=2 a - \frac{9}{2}  > 1$. 
Recall that $\lambda (G) \ge \lambda (T_{n,2,q}) > {n}/{2}$ 
and $b\le q-3$. 
Then for $n\ge 2(q^2+q)$, it follows that  
\[  \frac{x_{v^*}}{x_{w_1}} \le 
\frac{\lambda}{ \lambda - a}  < 1+ 
\frac{1}{b+3} < \frac{2a+b - {3}/{2}}{b+3}, \]  
which leads to $\bm{x}^{\mathrm{T}} \bm{y} 
(\lambda '-\lambda )\cdot  \lambda>0$ and 
$\lambda' > \lambda$, a contradiction. 
Thus,  $H_1$ is a star. 
\end{proof}

In Lemma \ref{lem-star}, we expect to prove that 
$H_i$ is a star, instead of a $4$-cycle. 
So we need to make 
an easy comparison between these two cases. 

\medskip 
If $n$ is odd, then $|S|= \frac{n-1}{2}$ 
and $|T|= \frac{n+1}{2}$. 
Since $G$ contains exactly $q\lfloor {n}/{2}\rfloor$ triangles, 
 the $q$ edges must be embedded into $T$. By Lemma \ref{lem-star}, we must embed 
a star $K_{1,q}$ or a cycle $C_4$ (in the case of $q=4$).  
By computation, embedding a copy of $K_{1,4}$ attains the maximum spectral radius. 
Indeed, we can compute that $\lambda (T_{n,2,4}) $ 
is the largest root of 
\[  t(x):= 9 - 10 n + n^2 + 4 x - 4 n x - (15 x^2)/4 - (n^2 x^2)/4 + x^4.  \]
Let $G$ be the graph obtained from  $T_{n,2}$ 
by embedding a $C_4$ into $T$. Then $\lambda (G)$ is the largest root of 
\[ g(x):= 7/2 - 4 n + n^2/2 + x/4 - (n^2 x)/4 - 2 x^2 + x^3. \]
It is not hard to check that $\lambda (G)< \lambda (T_{n,2,4})$. 
This completes the proof for odd $n$.

\medskip 
If $n$ is even, then $|S|= |T| = {n}/{2}$. 
Similarly, the embedded graphs $H_1$ and 
$H_2$ are stars or $4$-cycles by Lemma \ref{lem-star}. 
The remaining task is to show that 
embedding only one star into $S$ or $T$ 
will lead to a graph with the maximum spectral radius.

\begin{lemma} \label{lem-x-S} 
Let $n$ be even and $G'$ be the graph obtained from $T_{n,2}$ by embedding a star $K_{1,q}$ into the part $S$.  
Let $\lambda ' = \lambda (G')$ and $\bm{y}$ be the Perron vector of $G'$. 
Then 
\[  \frac{\lambda' y_V}{\lambda' + |S| + \frac{2q}{\lambda' -q}} \le y_T \le  \frac{\lambda' y_V}{\lambda' + |S| + \frac{2q}{\lambda' }} . \]
\end{lemma}

\begin{proof}
By definition, we have $y_V= y_S + y_T$. It is sufficient to prove that 
\[  |S| y_T + 2q \cdot \frac{y_T}{\lambda'} \le 
 \lambda' y_S \le |S| y_T + 2q\cdot  \frac{y_T}{\lambda' -q}.  \]
 For each $v\in S$, we have 
 $\lambda' y_v = y_T + \sum_{u\in N_S(v)} y_u$. Then 
 \begin{eqnarray*}
 \lambda' y_S = \sum_{v\in S} \lambda' y_v 
 = |S| y_T + \sum_{v\in S} \sum_{u\in N_S(v)} y_u 
 = |S|y_T + \sum_{u\in S} d_S(u) y_u.
 \end{eqnarray*}
 Note that $\lambda' y_u \ge y_T$ for each $u\in S$. Then 
 \[  \min_{u\in S} \{ y_u \} \ge \frac{y_T}{\lambda'}. \]  
We denote $y_{v^*}=\max_{u\in S}\{ y_u \}$. 
Similarly, we have $\lambda ' y_{v^*} = 
y_T + \sum_{w\in N_S(v^*)} y_w \le y_T + d_S(v^*) y_{v^*} = 
y_T + qy_{v^*}$, 
which yields $y_{v^*}\le \frac{y_T}{\lambda' -q}$. 
Thus, it follows that 
\[  |S|y_T + 2e(S) \frac{y_T}{\lambda' }  \le 
\lambda' y_S \le |S|y_T + 2e(S) \frac{y_T}{\lambda' -q}.  \]
The desired inequality holds immediately 
by  $e(S)= e(K_{1,q}) = q$. 
\end{proof}

\begin{lemma}
There exists exactly one $i\in \{1,2\}$ such that 
$H_i\neq \varnothing$ and 
$H_i\cong K_{1,q}$. 
\end{lemma}

\begin{proof} 
Suppose on the contrary that 
 both $H_1$ and $H_2$ are non-empty. 
By Lemma \ref{lem-star}, 
we know that $H_i$ is a star or a $4$-cycle. 
Thus, we have the following three cases: 
\begin{itemize}
\item[(a)]  $H_1\cong K_{1,q_1}$ and $H_2\cong 
K_{1,q_2}$, where $q_1+q_2=q$.  

\item[(b)] $H_1\cong C_4$ and $H_2\cong K_{1,q_2}$, where $4+q_2 =q$. 

\item[(c)]  $H_1\cong H_2\cong C_4$ in the case of $q=8$.

\end{itemize} 
By direct calculation, it is not hard to verify that $\lambda (G)< \lambda (T_{n,2,q})$. 
Next, we finish the proof in another way by using the double-eigenvector technique. 
We define $G^{\prime}$ to be the graph 
 obtained from $G$ by deleting all edges of $H_1$ and $H_2$, and adding a copy $H\cong K_{1,q}$ into $S \backslash V(H_1)$.  
 To obtain a contradiction, it suffices to show $\lambda (G') > \lambda (G)$. 
Recall that $x_S = \sum_{u\in S} x_u$ and $x_T=\sum_{u\in T} x_u$. 
Without loss of generality, 
 we may assume  that $x_{T} \geq x_{S}$. 
 Let $\lambda^{\prime}=\lambda\left(G^{\prime}\right)$ and $\bm{y}=\left(y_1, y_2, \ldots, y_n\right)^{\mathrm{T}}$ be the Perron vector of $G^{\prime}$. 
 To show $\lambda' > \lambda $, we shall prove 
 \[  \bm{x}^{\mathrm{T}}\bm{y} (\lambda '- \lambda) 
 \cdot \lambda \lambda ' >0. \]
 Since $G$ is the graph with the maximum spectral radius, 
 we have 
 $\lambda \geq \lambda^{\prime}  > \frac{n}{2}$.
 Observe that $H\cong K_{1,q}$ and 
$\sum_{v\in V(H)} d_H^2(v)= q^2+q$. 
For each $v\in S$, we have 
 \[  \lambda ' y_v = y_T + \sum_{u\in N_H(v)} x_u \ge 
 y_T + d_H(v) \cdot \min_{u\in S} \{ y_u \}. \] 
As $\lambda' y_u \ge y_T$ for any $u\in S$, we get 
$\min_{u\in S} \{ y_u\} \ge \frac{y_T}{ \lambda'}$.  
Using the double counting gives 
\begin{eqnarray*} 
\sum_{\{u,v\} \in E(H)} \lambda^{\prime} (y_u+y_v ) 
&= &  \sum_{v\in V(H)} d_{H}(v) \cdot \lambda ' y_v  \\
&\geq & \sum_{v \in V(H)} d_H(v) y_{T}+\sum_{v \in V(H)} d_H^2(v) \cdot \min _{u \in S} \{y_u\} \\ 
&\ge & 2 q y_T + (q^2+q ) \frac{y_T}{\lambda^{\prime}}. 
\end{eqnarray*} 
Note that $\lambda \geq \lambda^{\prime}$ 
and $V(H) \subseteq S \backslash V(H_1)$ 
is an independent set in $G$. 
Then $N_{G}(v)=T$ and $\lambda x_v=x_{T}$
 for each $v \in V(H)$.  
 It follows that
\begin{equation}  \label{eq-11}
\sum_{\{u, v\} \in E(H)} \lambda \lambda^{\prime}\left(x_u y_v+x_v y_u\right)=\sum_{\{u, v\} \in E(H)} \lambda^{\prime}\left(y_u+y_v\right) x_{T} 
\geq \left(2 q+\frac{q^2+q}{\lambda}\right) y_{T} x_{T} .
\end{equation}
Since $H_1$ is a star or a $4$-cycle, 
we have $\sum_{v \in V\left(H_1\right)} d_{H_1}^2(v) \le  q_1^2+q_1$. 
Let $x_{v^*} =\max\{x_u: u\in S\}$. 
Note that $\lambda x_{v^*} = x_T + \sum_{u\in N_{H_1}(v^*)} x_u \le x_T + q x_{v^*}$. Then 
\[   x_{v^*}  \leq \frac{x_{T}}{\lambda- q}. \]  
Therefore, we have
\begin{eqnarray*}  
\sum_{\{u, v\} \in E(H_1)} \lambda\left(x_u+x_v\right)  
&= & \sum_{v\in V(H_1)} d_{H_1}(v) \cdot \lambda x_v  \le  
\sum_{v\in V(H_1)} d_{H_1}(v) \bigl( x_T + d_{H_1}(v) x_{v^*} \bigr) \\
& \leq & \sum_{v \in V(H_1)} d_{H_1}(v) x_{T} 
+\sum_{v \in V(H_1)} d_{H_1}^2(v) x_{v^*}  \\
&\leq & 2 q_1 x_T 
+ (q_1^2+q_1 ) \frac{x_T}{\lambda-q} .
\end{eqnarray*} 
Note that $V(H_1)$ is an independent set of the new graph $G^{\prime}$.  
For every $v \in V(H_1)$, 
we have $N_{G^{\prime}}(v) =T$ and  $\lambda^{\prime} y_v=y_{T}$. Thus, we obtain
\begin{equation} \label{eq-22}
\sum_{\{u, v\} \in E(H_1)} \lambda \lambda^{\prime}\left(x_u y_v+x_v y_u\right) 
= \sum_{\{u ,v\} \in E(H_1)} \lambda\left(x_u+x_v\right) \cdot y_{T}  
\leq  \left(2 q_1+\frac{q_1^2+q_1}{\lambda- q}\right) x_T y_T.
 \end{equation}
Similarly, we can get $\max _{u \in T} x_u \leq \frac{x_S}{\lambda- q}$ and 
\begin{equation} \label{eq-33} 
\sum_{\{u, v\} \in E(H_2)} \lambda \lambda^{\prime}\left(x_u y_v+x_v y_u\right) 
\leq  \left(2 q_2+\frac{q_2^2+q_2}{\lambda- q}\right) x_S y_S.
  \end{equation}  
Considering the above inequalities, we can obtain 
\begin{eqnarray*}  
&& \bm{x}^{\mathrm{T}} \bm{y} 
(\lambda^{\prime}-\lambda )\cdot \lambda \lambda^{\prime}\\[3mm]
 & =& \sum_{\{u, v\} \in E(H)} \lambda \lambda^{\prime}\left(x_u y_v+x_v y_u\right)-\sum_{\{u, v\} \in E\left(H_1\right) \cup E\left(H_2\right)} \lambda \lambda^{\prime}\left(x_u y_v+x_v y_u\right) \\
& \overset{(\ref{eq-11},\ref{eq-22},\ref{eq-33})}{\geq} & \left(2 q+\frac{q^2+q}{\lambda}\right) x_T y_T 
- \left(2 q_1+\frac{q_1^2+q_1}{\lambda- q}\right) x_{T} y_T 
- \left(2 q_2+\frac{q_2^2+q_2}{\lambda- q}\right) x_{S} y_S \\
& = & \left(2 q_2+\frac{q^2+q}{\lambda}-\frac{q_1^2+q_1}{\lambda-q}\right) x_T y_ T 
-\left(2 q_2+\frac{q_2^2+q_2}{\lambda-q}\right) x_S y_S .
\end{eqnarray*} 
To show the right hand side is positive, 
we denote 
\[ c_1:=2 q_2+\frac{q^2+q}{\lambda} 
-\frac{q_1^2+q_1}{\lambda-q} \] 
 and 
 \[ c_2 :=2 q_2+\frac{q_2^2+q_2}{\lambda-q}. \] 
 By computation, we can check that $q_1q_2\ge q-1$ and 
 \[  c_1-c_2=\frac{2q_1q_2 \lambda - q^3-q^2}{\lambda (\lambda -q)} > \frac{(q-1)n - q^3-q^2}{(n/2 +q)n/2} > \frac{1}{5n}, \]
 where we used the assumption 
 ${n}/{2} < \lambda < {n}/{2} +q$ 
 and $n\ge 121q^2$. 
 Observe that $T$ is an independent set 
 in $G'$. Then for each $u\in T$, we have 
 $\lambda' y_u = y_S$.   Summing over all $u\in T$ yields 
 $|T| y_S = \lambda' y_T= \lambda' (y_V -y_S)$. 
Thus we get  
 \[  y_S= \frac{\lambda' y_V}{\lambda ' + |T|} . \]  
Using Lemma \ref{lem-x-S} and $|S|=|T|= \frac{n}{2}$,
it follows that  
  \begin{eqnarray*}  
  y_S-y_T \le  
  \frac{\lambda' y_V}{\lambda ' + |T|} - \frac{\lambda' y_V}{\lambda' + |S| + \frac{2q}{\lambda' -q}} 
  <  \frac{2q}{n^2(n/2 -q)} \lambda^{\prime} y_V. 
 \end{eqnarray*}  
Combining with the assumption that $x_T \geq x_S$, 
  we get 
\begin{eqnarray*}   
&& \bm{x}^{\mathrm{T}} \bm{y} 
(\lambda^{\prime}-\lambda) \cdot  \lambda \lambda^{\prime}  \\[3mm]  
&\ge &  c_1x_Ty_T - c_2 x_S y_S   \\[3mm]
& \geq & \left(c_2+\frac{1}{5n}\right)  x_S 
\left(y_S -\frac{2q}{n^2(n/2 -q)} \lambda^{\prime} y_V \right)  -c_2 x_S y_S \\
& > & \frac{ x_S}{5n} \cdot \frac{\lambda' y_V}{n+q}  - 
 ( 2q-1)x_S \cdot \frac{2q}{n^2(n/2 -q)} \lambda^{\prime} y_V , 
\end{eqnarray*}  
where the last inequality holds by 
 $y_S = \frac{\lambda' y_V}{\lambda ' + |T|} > \frac{\lambda' y_V}{n+q}$ and 
$c_2  < 2(q-1) + \frac{q^2}{n/2 -q}$. 
Hence,  we can obtain $\bm{x}^{\mathrm{T}} \bm{y} 
(\lambda^{\prime}-\lambda) \cdot  \lambda \lambda^{\prime} >0$ for $n\ge 121q^2$. This a contradiction 
with $\lambda^{\prime}>\lambda$. 
Thus, there exists exactly one $i\in \{1,2\}$ such that 
$H_i\neq \varnothing$ and then $H_i\cong K_{1,q}$. 
\end{proof}

\section{Proof of Theorem \ref{spectral-BC}}

\label{sec6}

Now, we are in a position to 
 prove Theorem \ref{spectral-BC}. 

\begin{proof}[{\bf Proof of Theorem \ref{spectral-BC}}]
Assume that $G$ is a graph on $n$ vertices with
 $\lambda (G)\ge \lambda (T_{n,2})$ and $\tau_3(G)\ge s$, 
 where $s\ge 1$.  If $s=1$, then 
 $G\neq T_{n,2}$. Theorem \ref{thmNZ2021} implies $G$ has 
 $\lfloor \frac{n}{2} \rfloor -1$ triangles. 
 For $s=2$, the authors \cite{LFP-count-bowtie} proved that $G$ has at least $n-3$ triangles. 
Now, we consider the case $s\ge 3$. 
We may assume that $t(G)\le \frac{1}{2} s{(n-1)}$. 
 Otherwise, we are done. 
 
 \medskip 
\noindent 
{\bf Claim 1.} 
We have $e(G)> \lfloor n^2/4\rfloor -3s$. 
There is a partition $V(G) =S\cup T$ such that 
$e(S) + e(T)< 6s$, $e(S,T)>\lfloor n^2/4\rfloor -9s$ and 
$ \frac{n}{2} - 3\sqrt{s} \le |S|, |T| \le \frac{n}{2} +3\sqrt{s}$. 

\begin{proof}[Proof of Claim 1]
 Since $\lambda (T_{n,2}) = \sqrt{\lfloor n^2/4\rfloor} 
 > \frac{1}{2} (n-1)$, by 
 Lemma \ref{thm-BN-CFTZ-NZ}, we have 
\[   e(G)\ge \lambda^2- \frac{3t}{\lambda} 
>  \lambda^2 - \frac{6t}{n-1} \ge   
\Big\lfloor \frac{n^2}{4} \Big\rfloor -3s . \]  
If $G$ is $6s$-far from being bipartite, 
then using Theorem \ref{thm-far-bipartite}, we get 
\[ t(G)\ge \frac{n}{6} 
\left(e(G) + 6s - \frac{n^2}{4} \right) > s\frac{n}{2}, \]  
which contradicts with the assumption.  
 Thus, $G$ is not $6s$-far from being bipartite. 
Namely, there exists a vertex partition 
$V(G)=S\cup T$ such that 
\[ e(S) + e(T)< 6s. \]  
Consequently, we get 
\[  e(S,T) = e(G) - e(S) - e(T) > \Big\lfloor \frac{n^2}{4} \Big\rfloor - 9s. \]  
By the symmetry, we may assume that $|S|\le |T|$. 
Suppose on the contrary that $|S|< \frac{n}{2} - 3\sqrt{s}$. 
Then $|T| = n-|S| > \frac{n}{2} +3\sqrt{s} $. 
It follows that $e(S,T) \le |S||T| \le  \lfloor n^2/4\rfloor -9s$, 
which is a contradiction. Thus, we have 
$ \frac{n}{2} - 3\sqrt{s} \le |S|, |T| \le \frac{n}{2} +3\sqrt{s}$.   
\end{proof}

Next, we are going to refine the above bipartition. 

\medskip 
\noindent 
{\bf Claim 2.} We have $e(S) + e(T)=s$. 

\begin{proof}[Proof of Claim 2]
Observe that every triangle of $G$ must contain an edge from $G[S]$ or $G[T]$. If $e(S) + e(T)\le s-1$, then 
all triangles of $G$ can be covered by a vertex set of 
at most $s-1$ vertices, contradicting with $\tau_3(G)\ge s$. 
Now, suppose that $e(S) + e(T)=k$ for some integer $k\ge s+1$.  Note that  
each edge of $G[S]$ or $G[T]$ can yield 
at least ${n}/{2} -3\sqrt{s}$ triangles. 
From Claim 1, we know that  $|S||T| - e(S,T)\le \lfloor n^2/4\rfloor - e(S,T)<9s$. Then 
$G[S,T]$ misses at most $(9s -1)$ edges between 
$S$ and $T$,  and each missing edge  destroys at most $k$ triangles. 
So we get $t(G)\ge k(\frac{n}{2} -3\sqrt{s}) - (9s-1)k \ge k (\frac{n}{2} - 3\sqrt{s} -9s+1) > \frac{1}{2}s(n-1)$, 
where the last inequality holds for every $s\ge 3$ and $n\ge 28s^2$. 
This contradicts with the assumption $t(G)\le \frac{1}{2}s(n-1)$. 
Thus, we have $e(S) + e(T)=s$, as claimed. 
\end{proof}

From the above claims,  
we obtain 
$e(S,T)= e(G) - s > \lfloor n^2/4\rfloor - 4s$. Then 
\[ \frac{n}{2} - 2 \sqrt{s} \le |S|, |T| \le \frac{n}{2} +2 \sqrt{s}. \]  
If $s=3$, then $\lceil \frac{n}{2} \rceil -3\le 
|S|,|T|\le \lfloor \frac{n}{2} \rfloor +3$.  
Since there are less than $4s$ missing edges in $G[S,T]$, Claim 2 implies that 
$ t(G)\ge s(\lceil \frac{n}{2} \rceil -3) -4s \cdot s 
\ge \frac{1}{2}sn - 5s^2$. 
If $s\ge 4$, then a similar argument shows that $t(G)\ge s(\frac{n}{2} - 2\sqrt{s}) - 4s^2 \ge \frac{1}{2}sn - 5s^2$, as needed. 
\end{proof}

\section{Concluding remarks}

\label{sec7}

Recall that $t(G)$ denotes the number of triangles in a graph $G$. 
A special case of an aforementioned 
result of Bollob\'{a}s and Nikiforov \cite{BN2007jctb} 
states that 
\[  t(G) \ge \frac{n^2}{12}\left( \lambda -\frac{n}{2}\right). \]
In fact, a careful examination indicates that the above inequality holds strictly. Motivated by this observation, 
one may ask the following problem.

\begin{problem} 
What is the best constant $C>0$ such that $ t(G) \ge C n^2 \left(\lambda - \frac{n}{2} \right)$?
\end{problem}

\noindent 
{\bf Remark.} 
By the clique density theorem of Reiher \cite{Rei2016}, 
we know that if $\varepsilon >0$ and $G$ is a graph on $n$ vertices with $e(G)\ge (\frac{1}{4} + \varepsilon)n^2$, then $G$ has at least 
$(\frac{1}{2} + o(1)) \varepsilon n^3$ triangles, where $o(1)\to 0$ whenever 
$\varepsilon \to 0$. With the aid of this powerful result, 
one can easily prove that if $G$ has at least $(\frac{1}{4} + \varepsilon )n^2$ edges, where we set $\varepsilon := \frac{1}{2}(\frac{\lambda}{n} - \frac{1}{2})$, then $G$ has at least $(\frac{1}{4} + o(1))n^2(\lambda - \frac{n}{2})$ triangles. Moreover, the graph $T_{n,2,1}$ shows 
that the coefficient $\frac{1}{4}$ is achievable. 
Recently, Shengtong Zhang told us that taking $G=T_{n,3}$ yields $C\le \frac{2}{9}$.

\medskip 
Another question as to how many triangles appear at least in any $n$-vertex graph with given spectral radius turns out to be interesting.
For instance, for any graph $G$ with 
$\lambda (G)= \gamma n$, where $\gamma \in (\frac{1}{2}, \frac{2}{3}]$, that is, $\lambda (T_{n,2})<\lambda (G) \le \lambda (T_{n,3})$, 
  there might exist a graph, say $L$, whose shape is approximately between $T_{n,2}$ and $T_{n,3}$   
 such that $\lambda(L) = \gamma n$ and $t(L)\le t(G)$. 
Formally, we now define the minimizer $L$ as follows.  
For $s\in \mathbb{N}$ and $\alpha >0$, 
we define $L_{n,s,\alpha}$ as a complete multipartite graph with 
$V(L_{n,s,\alpha})=V_1 \cup \cdots V_s \cup V_{s+1}$, 
where $|V_1|=\cdots =|V_s|= \frac{n(1+ \alpha )}{s+1}$ and 
$|V_{s+1}|= \frac{n(1-s\alpha)}{s+1}$. 
Next, we speculate that $L_{n,s,\alpha}$ 
contains the minimum number of triangles 
among all graphs with given spectral radius.

\begin{conjecture}[Spectral triangle density conjecture]
Let $\gamma \in (\frac{1}{2},1)$ 
and $G$ be an $n$-vertex graph with  $\lambda (G)\ge \gamma n$. 
Then $t(G)\ge t(L_{n,s,\alpha})$, where 
 $s\ge 2$ is an integer with 
$\gamma \in (\frac{s-1}{s},\frac{s}{s+1}]$ and 
 $\alpha>0 $ is a maximal number such that 
$\lambda (L_{n,s,\alpha}) \ge \gamma n$. 
\end{conjecture}

 Nosal \cite{Nosal1970}  
proved that every graph $G$ with
$m$ edges and $\lambda (G) > \sqrt{m}$ 
contains a triangle; see, e.g., \cite{Niki2002cpc,Ning2017-ars}. 
In 2023, Ning and Zhai \cite{NZ2021}
showed that if
$  \lambda (G)\ge \sqrt{m}$,
then $t(G) \ge \lfloor \frac{\sqrt{m}-1}{2} \rfloor$, unless $G$ is a  complete bipartite graph. 
We propose the following conjecture. 

\begin{conjecture} \label{conj-72}
If $G$ is a graph with $m$ edges and 
\[  \lambda (G)\ge \frac{1+\sqrt{4m-3}}{2}, \]  
then $G$ has at least $\frac{m-1}{2}$ triangles, 
with equality if and only if $G=K_2\vee \frac{m-1}{2}K_1$. 
\end{conjecture}

\noindent 
{\bf Remark.} 
Roughly speaking, Conjecture \ref{conj-72} reveals that 
in a graph $G$ with $\lambda (G) \approx \sqrt{m} + \frac{1}{2}$, 
the number of triangles 
jumps from $\frac{1}{2}\sqrt{m}$  
to $ \frac{1}{2}m$. 
In addition, we can see that if $\lambda (G)\ge (1+ \varepsilon)\sqrt{m}$, 
then $t(G)> \frac{2\varepsilon}{3}m^{3/2}$. 
In fact, it is not difficult to obtain a weaker bound than Conjecture \ref{conj-72}. Indeed, it follows from Lemma \ref{thm-BN-CFTZ-NZ} that 
$t(G)> \frac{1}{3}\lambda (\lambda^2-m) \ge \frac{m-1}{3}$. 
Thus, the difficulty of Conjecture \ref{conj-72} lies in 
deriving the tight bound $t(G)\ge \frac{m-1}{2}$.

\medskip 
Theorem \ref{thm-LS1975} asserts that 
 every graph $G$ with $\lfloor n^2/4\rfloor +q$ 
 edges contains at least $q \lfloor n/2\rfloor$ triangles for every $q< n/2$. 
In Theorem \ref{thm-Yn2q}, 
we proved a spectral version for all $q\le \frac{1}{11}\sqrt{n}$. 
This bound on $q$ is tight up to a constant factor (Example \ref{exam}). 
Under our framework, it seems feasible to improve the constant factor slightly. However, new ideas are needed to determine the exact bound on $q$. It is ambitious to propose the following problem. 

\begin{problem}  \label{prob7-4}
 What is the supremum $\delta \in (0,0.8)$ such that Theorem \ref{thm-Yn2q} holds for every integer $q< \delta \sqrt{n}$? Additionally, what is the optimal range of $q$ in Theorem \ref{thm-sp-LS}?  
 \end{problem}

In 2017, 
 Pikhurko and Yilma \cite{PY2017} 
 proved that for every color-critical graph 
 $F$ with $\chi (F)=r+1$, 
there exist $\delta =\delta (F)>0$
and $n_0$ such that
for all $n\ge n_0$  and $q < \delta n$,
if $G$ is an $n$-vertex graph
with $e(T_{n,r}) +q$ edges and it achieves the minimum 
number of copies of $F$, 
then $G$ is obtained from the Tur\'{a}n graph $T_{n,r}$ by adding $q$ new edges.  Furthermore, 
by Theorem \ref{thm-LS1975}, we see that $\delta (K_3)= \frac{1}{2}$. A result of Lov\'{a}sz and Simonovits \cite{LS1975,LS1983} implies that $\delta(K_{r+1})=\frac{1}{r}$ for every $r\ge 2$.  
Moreover, Pikhurko and Yilma \cite{PY2017} also determined the exact value of $\delta$ 
for a number of color-critical graphs, such as,  odd cycles $\delta (C_{2k+1}) =\frac{1}{2}$, and 
cliques with one edge removed $\delta (K_{r+2} -e) = \frac{r-1}{r^2}$.  Similarly, it is also meaningful to determine the exact value of $\delta (F)$ for various graphs $F$ in the spectral setting.

\section*{Acknowledgements} 

We are grateful to Xiaocong He and Loujun Yu for carefully reading an early draft of the paper.  
Yongtao Li would like to thank Dheer Noal Desai for inspiring discussions on spectral graph problems. 
 Lihua Feng was supported by the National Natural Science Foundation of China (Nos. 12271527 and 12471022), and Yuejian Peng was supported by 
 the National Natural Science Foundation of Hunan Province (No. 2025JJ30003) and 
the National Natural Science Foundation of China (No. 12571363).


\end{document}